\documentclass{article}

\usepackage{indentfirst}
\usepackage{amsmath,amssymb,amsthm}
\usepackage[margin=2cm]{geometry}
\usepackage{mathrsfs}
\usepackage{bbold}
\usepackage{comment}

\newtheorem{theorem}{Theorem}[section]
\newtheorem{definition}{Definition}[section]
\newtheorem{proposition}{Proposition}[section]
\newtheorem{lemma}{Lemma}[section]
\newtheorem{corollary}{Corollary}[section]
\newtheorem{remark}{Remark}[section]

\newcommand{\R}{\mathbb{R}}
\newcommand{\hl}{\mathbb{R}_{\geq0}}
\newcommand{\abs}[1]{\left\vert #1\right\vert}
\newcommand{\norm}[1]{\left\vert\left\vert #1\right\vert\right\vert}

\title{Balanced homogeneous harmonic maps between cones}
\author{Brian Freidin\\Auburn University\\bgf0012@auburn.edu}
\date{}

\begin{document}

\maketitle

\begin{abstract} We study the degrees of homogeneous harmonic maps between simplicial cones. Such maps have been used to model the local behavior of harmonic maps between singular spaces, where the degrees of homogeneous approximations describe the regularity of harmonic maps. In particular the degrees of homogeneous harmonic maps are related to eigenvalues of discrete normalized graph Laplacians.
\end{abstract}

\section{Introduction}

The study of harmonic maps into singular spaces goes back to the seminal work \cite{gromov-schoen} of Gromov-Schoen, who considered maps into Riemannian simplicial complexes. Of particular interest to the present work is their local approximation of such maps by \emph{homogeneous} maps between tangent cones. The study of harmonic maps was extended by Chen in \cite{chen} to include simplicial domains, and the theory for simplicial domains was further studied in \cite{eells-fuglede}, \cite{mese}, \cite{daskal-mese-1}, \cite{daskal-mese-2}, among others. The H\"older continuity of harmonic maps from Riemannian simplicial complexes to metric spaces of non-positive curvature was established in \cite{eells-fuglede}. Further regularity away from the $(n-2)$-skeleton of an $n$-dimensional simplicial domain was studied in \cite{mese}, \cite{daskal-mese-1}, \cite{daskal-mese-2}.

The blow-up analysis of \cite{gromov-schoen} was emulated by Daskalopoulos and Mese in \cite{daskal-mese-1}, \cite{daskal-mese-2} to construct tangent maps as homogeneous approximations of harmonic maps with singular domains and targets. The general idea is to fix a point $p$ in the domain of a harmonic map $u$, and its image $u(p)$. For $\lambda,\mu>0$ one constructs the map
\[
    u_{\lambda,\mu}(x) = \mu^{-1}u(\lambda x).
\]

A small neighborhood of $p$ in the domain, scaled by the factor $\lambda^{-1}$, serves as the domain of $u_{\lambda,\mu}$. The map first rescales the neighborhood to its original size. Then after mapping by the original map $u$, the image is scaled by a factor of $\mu^{-1}$. If $\lambda$ and $\mu$ are small this has the effect of magnifying small neighborhoods of $p$ and $u(p)$. As $\lambda,\mu\to0$, technical estimates on the relationship between $\lambda$ and $\mu$ ensure the existence of a limit
\[
    u_* = \lim_{\lambda,\mu\to0}u_{\lambda,\mu}.
\]
The limit map $u_*$ has as domain the tangent cone to the domain of $u$ at the point $p$, and has as target the tangent cone at $u(p)$ to the target of $u$. The tangent cone to a smooth Riemannian manifold is just its Euclidean tangent space, while the tangent cone to a simplicial complex is a geometric cone over a small sphere around a point, or its link. The map $u$ is homogeneous in the sense that for each $x$ in the domain cone, $t\mapsto u_*(tx)$ is an (unparametrized) geodesic, and there is some $\alpha>0$ so that $\abs{u_*(tx)} = t^\alpha\abs{u_*(x)}$. Here $\abs{\cdot}$ denotes the distance from the vertex of the target cone. For more details on the blow-up procedure and homogeneous harmonic maps, see the references cited above, especially \cite{gromov-schoen} and \cite{daskal-mese-1}. A typical application of this blow-up procedure is to say that if at each point $p$ in the domain, the corresponding blow-up map $u_*$ has degree $\alpha\geq1$, then the original map $u$ is locally Lipschitz continuous.

In related work (\cite{me-gras-1}, \cite{me-gras-2}) the present author, together with Vict\`oria Gras Andreu, consider a class of harmonic maps between 2-dimensional simplicial complexes. The maps respect the simplicial structure of the domain and target in the sense that vertices (resp. edges, faces) of the domain are mapped to vertices (resp. edges, faces) of the target. In \cite{me-gras-1} the simplicial complexes are endowed with metrics identifying each face with the ideal hyperbolic triangle. This has the result of putting the vertices of the complexes at infinite distance, effectively puncturing the spaces. In \cite{me-gras-2} the complexes are endowed with simplex-wise flat metrics, as well as metrics conformal to flat or ideal hyperbolic metrics. The blow-up procedure described above is repeated in \cite{me-gras-2} to construct tangent maps at edge points and at (non-punctured) vertices.

In \cite{daskal-mese-2} the tangent maps of harmonic maps from 2-dimensional simplicial complexes into Riemannian manifolds are related to eigenfunctions of a discrete Laplace operator on the link $Lk(p)$ of a point $p$. The link of a point $p$ in a 2-complex $X$ is a graph whose vertices correspond to edges of $X$ containing $p$, and whose edges correspond to faces of $X$ containing $p$. The strategy of relating continuous and discrete harmonic maps is explored in detail in \cite{daskal-mese-2} and several references contained therein. The discrete Laplace operator considered in \cite{daskal-mese-2} assumes the geometry in the link $Lk(p)$ is very regular, in particular that all faces of the 2-complex are equilateral triangles so that every edge of $Lk(p)$ has the same length. The main goals of this paper are to extend that analysis to include more of the richness of the geometry that $Lk(p)$ can have, as well as to the setting of tangent maps between simplicial cones constructed in \cite{me-gras-2}.

Before beginning with the main content, we discuss some history of discrete Laplace operators that will become useful later on. For a more detailed discussion, including proofs of quoed results, see e.g. \cite{chung} or \cite{banerjee-jost}. Future work may be done to generlize the results of this paper to higher dimensional cones, at which point the work of \cite{horak-jost} on the spectra of simplicial complexes will likely play the role of the discussion below.

Let $\Gamma$ be a graph with vertices $V=V(\Gamma)$ and edges $E = E(\Gamma)$. One can enumerate the vertices of $\Gamma$ and thus identify functions $f:V\to\R$ with vectors in $\R^{\#V}$. Then the unweighted graph Laplacian is represented by the matrix $\Delta$ with entries
\[
    \Delta_{ij} = \begin{cases} deg(v_i) & i=j\\ -1 & v_i\sim v_j\\ 0 & \text{else}.\end{cases}
\]
Here $v_i\sim v_j$ means that the vertices $v_i$ and $v_j$ are adjacent in $\Gamma$. This matrix $\Delta=D-A$ is the difference between the diagonal degree matrix $D$ and the adjacency matrix $A$ with entries
\[
    D_{ij} = \begin{cases} deg(v_i) & i=j\\ 0 & i\neq j\end{cases} \qquad\text{and}\qquad A_{ij} = \begin{cases}1 & v_i\sim v_j\\ 0 & v_i\not\sim v_j.\end{cases}
\]
A normalized Laplace operator, denoted $\mathcal{L}$, is defined by
\[
    \mathcal{L} = Id - D^{-1/2}AD^{-1/2} = D^{-1/2}\Delta D^{-1/2}.
\]

Both the matrices $\Delta$ and $\mathcal{L}$ are symmetric and positive semi-definite. Indeed, they determine quadratic forms
\[
    \rho\cdot\Delta\rho = \sum_{uv\in E}\Big(\rho(v)-\rho(u)\Big)^2 \qquad\text{and}\qquad \rho\cdot\mathcal{L}\rho = \sum_{uv\in E}\left(\frac{\rho(v)}{\sqrt{deg(v)}} - \frac{\rho(u)}{\sqrt{deg(u)}}\right)^2.
\]
An immediate consequence is that the vector $w=\mathbb{1}$ with $w_i=1$ for all $i$ is in the kernel of $\Delta$ and in fact spans the kernel if $\Gamma$ is connected. Since $\mathcal{L} = D^{-1/2}\Delta D^{-1/2}$, the vector $D^{1/2}\mathbb{1}$ spans the kernel of $\mathcal{L}$. The operator $\mathcal{L}$ is called normalized because its largest eigenvalue is $\leq 2$. The eigenvalue $\lambda=2$ of $\mathcal{L}$ is achieved if and only if $\Gamma$ is bipartite, and the first positive eigenvalue of $\mathcal{L}$ is at most $1$ unless $\Gamma$ is a complete graph $K_n$, in which case the first eigenvalue is $\frac{n}{n-1}$.

If weights are assigned to $\Gamma$ by a function $w:V\cup E\to\hl$. Then one can modify the adjacency matrix so that $A_{ij} = w(v_iv_j)$ when $v_i\sim v_j$, and the degree matrix so that $D_{ii} = w(v_i)$. In case only edge weights are provided, the vertex weights are determined by $w(v_i) = \sum_{v_j\sim v_i}w(v_iv_j)$. In each of these cases one can still define
\[
    \Delta = D-A\qquad\text{and}\qquad\mathcal{L}=Id-D^{-1/2}AD^{-1/2} = D^{-1/2}\Delta D^{-1/2}.
\]
In case $w(e)=1$ for all $e\in E$ and $w(v) = deg(v)$ for $v\in V$ the unweighted operators are recovered. The matrices are still positive semi-definite but their kernels are affected by the weight function.

In both \cite{chung-langlands} and \cite{banerjee-jost} the asymmetric matrix $L = D^{-1}\Delta$ is studied. This matrix also comes up in the present work. Fortunately it is similar to the normalized Laplacian via the similarity
\[
    L = D^{-1/2}\mathcal{L}D^{1/2}.
\]
As a result the spectrum of $L$ is identical to the spectrum of $\mathcal{L}$.

The structure of this paper is as follows. We begin in Section~\ref{S preliminaries} by introducing the spaces and maps to be studied. The spaces include $k$-pods and simplicial cones, together with the metrics they can carry. The maps of interest are homogeneous maps between cones. In Section~\ref{S preliminaries} we also introduce the notion of harmonic maps as well as the balancing conditions that are essential to the current study.

In Section~\ref{S euclidean target} we define and study balanced homogeneous harmonic maps from simplicial cones into Euclidean spaces. In Section~\ref{S k-pod target} we do the same for maps from Euclidean spaces into $k$-pods. Section~\ref{S k-pod target} also allows for a domain with a single singular point, as well as $p$-harmonic maps. In Section~\ref{S singular target} we define and study balanced homogeneous harmonic maps between simplicial cones. Finally in Section~\ref{S puncture} we generalize Sections~\ref{S euclidean target} and~\ref{S singular target} by studying the limits as the angles of each face of the cone tend to 0 and the vertices become punctures.

\section{Preliminaries}\label{S preliminaries}

\subsection{Spaces and metrics}

Our main objects of study will be homogeneous maps between cones. The spaces and maps we consider will have a lot more structure, but the basic definitions are quite general.

\begin{definition}[cone]\label{cone}
    For any topological space $X$, the cone $C(X)$ is the quotient
    \[
        C(X) = \big(X\times\hl\big)/\sim,
    \]
    where $(x,0)\sim(y,0)$ for all $x,y\in X$.
\end{definition}

If $X$ is a space of dimension $n$, then the cone $C(X)$ has dimension $n+1$. Thus one-dimensional cones can be constructed as the cone over a finite set of points. The common point $(x,0)\in C(X)$ is called the \emph{vertex} of the cone.

\begin{definition}[homogeneous map]\label{homogeneous}
    A map $f:C(X)\to C(Y)$ is homogeneous of degree $\alpha>0$ if there are functions $y:X\to Y$ and $s:X\to\hl$ so that
    \[
        f(x,t) = \big(y(x),s(x)t^\alpha\big)
    \]
\end{definition}

Necessarily the image of the vertex of $C(X)$ under a homogeneous map is the vertex of $C(Y)$.

The remainder of this section is devoted to describing which cones will appear as domains and targets of the maps we consider.

\begin{definition}[$k$-pod]
    For $k\in\mathbb{N}$, consider the space $\{x_1,\ldots,x_k\}$ of $k$ discrete points. The $k$-pod is the cone $C(k)=C(\{x_1,\ldots,x_n\})$. The rays $C(x_j) = \{x_j\}\times\hl$ are called the edges of the $k$-pod.
\end{definition}

Each edge of the $k$-pod $C(k)$ is homeomorphic to the half-line $\hl$. Pulling back the standard Euclidean metric from $\hl$ defines a metric on each edge of $C(k)$. We will also always use these homeomorphisms (now isometries) to give coordinates on each edge.

For a 2-dimensional cone, we must take the cone over a 1-dimensional space. The 1-dimensional spaces we will consider are graphs $\Gamma$, with vertices $V(\Gamma)$ and edges $E(\Gamma)$.

\begin{definition}[2-dimensional cone]\label{2-d}
    For a connected finite simplicial graph $\Gamma$, consider the cone $C(\Gamma)$. The rays $C(v)=v\times\hl$ for vertices $v\in V(\Gamma)$ are edges of the cone, and the sectors $C(e)$ for edges $e\in E(\Gamma)$ are faces of the cone.
\end{definition}

Each face of the cone $C(\Gamma)$ is homeomorphic to a sector of the plane $\R^2$. In order to determine the geometry of the cone one needs only specify the angle formed by each face at the vertex of the cone.

\begin{definition}[metrics]\label{metric}
    Given a cone $C(\Gamma)$ over a graph $\Gamma$, a function $\theta:E(\Gamma)\to(0,\pi)$ determines a metric. For $e\in E(\gamma)$ the face $C(e)$ is isometric to the sector $S_{\theta(e)}$ of the Euclidean plane given in polar coordinates by
    \[
        S_{\theta(e)} = \{(r,\theta) \vert r\geq0,0\leq\theta\leq\theta(e)\}.
    \]
    Let $C(\Gamma,\theta)$ denote the space $C(\Gamma)$ endowed with the metric determined by $\theta$.
\end{definition}

\begin{remark}
    One could also define metrics on $C(\Gamma)$ for functions $\theta:E(\Gamma)\to\R_{>0}$. For $0<\theta(e)<2\pi$ one can still model $C(e)$ by the sector $S_{\theta(e)}$. For $\theta\geq2\pi$ one could use a sector on a cone with cone angle $>2\pi$. Though most results are unchanged when considering this more general class of metrics, we will not consider them here for two main reasons. First, the simplicial cones considered here are meant to model tangent cones to Euclidean simplicial complexes, all of whose angles are at most $\pi$. And second, the analysis for $\theta\geq\pi$ is more subtle and distracts from the main flow of many arguments.
\end{remark}

Given a graph $\Gamma$ and an edge $e=(p,q)\in E(\Gamma)$, one can subdivide $e$ by adding a vertex $x$ in the interior. This produces a graph $\Gamma'$ with vertex set $V(\Gamma')=V(\Gamma)\cup\{x\}$ and edge set $E(\Gamma')=\Big(E(\Gamma)\backslash\{e\}\Big)\cup\{px,xq\}$.

Given a function $\theta:E(\Gamma)\to(0,\pi)$ and an edge $e\in E(\Gamma)$, when one subdivides the edge $e$ to produce $\Gamma'$ as above one can define a function $\theta':E(\Gamma')\to(0,\pi)$. For $e'\neq e$ set $\theta'(e')=\theta(e')$, and set $\theta'(px)=\alpha$ and $\theta'(xq)=\beta$ in any way so that $\alpha+\beta=\theta(e)$. Any pair $(\Gamma^*,\theta^*)$ obtained from $(\Gamma,\theta)$ by a finite sequence of edge subdivisions as in these two paragraphs is called a subdivision of $(\Gamma,\theta)$, and the cones $C(\Gamma^*,\theta^*)$ and $C(\Gamma,\theta)$ are isometric.

In general two cones $C(\Gamma_1,\theta_1)$ and $C(\Gamma_2,\theta_2)$ are isometric if and only if the graphs have a common subdivision $(\Gamma^*,\theta^*)$.

In $C(\Gamma,\theta)$ identify an edge $e\in E(\Gamma)$ with the interval $[0,1]$. Taking polar coordinates $(\rho,\varphi)$ in $S_{\theta(e)}$, a concrete isometry $\psi_e:C(e)=C([0,1])\to S_{\theta(e)}$ is given by

\[
    \psi_e(x,t)=(\rho,\varphi)\qquad\text{with}\qquad \rho = t,\quad \varphi = x\theta(e).
\]

We will also use these isometries to give coordinates on each face of $C(\Gamma,\theta)$. But note that the map $\psi_e(x,t) = (t,(1-x)\theta(e))$ also defines an isometry $C(e)\to S_{\theta(e)}$. Our choice of isometry will be determined by which endpoint of $e$ is most relevant at the moment.

\begin{definition}
    In a cone $C(\Gamma,\theta)$, fix a face $C(e)$ for some $e\in E(\Gamma)$. If $v\in V(\Gamma)$ is one endpoint of $e$, then the coordinates on $C(E)$ adapted to $v$ are the coordinates determined by the isometry $\psi_e:C(e)\to S_{\theta(e)}$ that sends the point $(v,1)\in C(e)$ to $(1,0)\in S_{\theta(e)}$.
\end{definition}

One can also view the Euclidean plane $\R^2$ as a 2-dimensional cone. Namely, for the cycle graph $C_n$ consisting of $n$ vertices and $n$ edges, $\R^2 = C(C_n,\theta)$ for any $\theta:E(C_n)\to(0,\pi)$ with $\sum_{e\in E(C_n)}\theta(e)=2\pi$. More generally, we will sometimes consider cones over the cycle graph with different metrics.

\begin{definition}
    A smooth cone is a geometric cone over the cycle graph. In other words, it is a space $C(C_n,\theta)$ for some $\theta:E(C_n)\to(0,\pi)$.
\end{definition}

The curvature at the vertex of a smooth cone is determined by the sum $\sum_{e\in E(C_n)}\theta(e)$. If the sum is less than $2\pi$ the vertex has positive curvature, if the sum is equal to $2\pi$ the vertex is flat, and if the sum is greater than $2\pi$ the vertex has negative curvature.

\subsection{Harmonic functions on sectors}

When it comes time to consider harmonic maps between cones, we will need to understand the structure of those harmonic maps. Our domains will always be 2-dimensional cones $C(\Gamma,\theta)$, all of whose faces are isometric to plane sectors $S_\theta$. To begin this section we first aim to understand harmonic functions $f:S_{\theta_0}\to\R$.

\begin{lemma}\label{triggy}
    A homogeneous function $u:S=S_{\theta_0}\to\R$ is harmonic if and only if there are constants $c_1,c_2$ so that
    \[
        u(r,\theta) = r^\alpha\Big(c_1\cos(\alpha\theta) + c_2\sin(\alpha\theta)\Big).
    \]
\end{lemma}

\begin{proof}
    In polar coordinates, a homogeneous function $u(r,\theta)$ of degree $\alpha$ has the form
    \[
        u(r,\theta) = r^\alpha u(1,\theta) = r^\alpha \rho(\theta).
    \]
    And in polar coordinates the Laplacian of $u$ reads
    \[
        \Delta u = \frac{\partial^2 u}{\partial r^2} + \frac{1}{r}\frac{\partial u}{\partial r} + \frac{1}{r^2}\frac{\partial^2 u}{\partial \theta^2} = r^{\alpha-2}\Big(\alpha^2\rho(\theta)+\rho''(\theta)\Big).
    \]

    To be harmonic, $u$ must satisfy $\Delta u=0$, and this equation must hold for all $(r,\theta)$. Hence $\rho$ must satisfy $\rho''(\theta)+\alpha^2\rho(\theta)=0$. The solutions to this ordinary differential equation are precisely
    \[
        \rho(\theta) = c_1\cos(\alpha\theta) + c_2\sin(\alpha\theta).
    \]
\end{proof}

In order to turn our study into a discrete problem later, we introduce the following result that says homogeneous harmonic functions are described by discrete data.

\begin{lemma}\label{BVP}
    A homogeneous harmonic function $u(r,\theta)=r^\alpha \rho(\theta):S=S_{\theta_0}\to\R$ is uniquely determined by the two numbers $\rho(0)$ and $\rho(\theta_0)$ unless $\alpha\theta_0\in\pi\mathbb{Z}$.
\end{lemma}

\begin{proof}
    Given the form of $\rho(\theta)$ from the Lemma~\ref{triggy}, one easily solves the boundary value problem on $S = S_{\theta_0}$.
    \begin{eqnarray*}
        \rho(0) & = & c_1\cos(0)+c_1\sin(0)\\
         & = & c_1.\\
        \rho(\theta_0) & = & c_1\cos(\alpha\theta_0) + c_2\sin(\alpha\theta_0)\\
         & = & \rho(0)\cos(\alpha\theta_0) + c_2\sin(\alpha\theta_0).
    \end{eqnarray*}
    So as long as $\alpha\theta_0\not\in\pi\mathbb{Z}$, we have $\sin(\alpha\theta_0)\neq0$ so we can write
    \[
        \rho(\theta) = \rho(0)\cos(\alpha\theta) + \frac{\rho(\theta_0)-\rho(0)\cos(\alpha\theta_0)}{\sin(\alpha\theta_0)}\sin(\alpha\theta).
    \]
    
    In case $\alpha\theta_0=n\pi$, $c_2$ can be any real number but we must have
    \[
        \rho(\theta_0) = \rho(0)\cos(n\pi) = (-1)^n \rho(0).
    \]
    In this case $\rho$ has the form
    \[
        \rho(\theta) = \rho(0)\cos(\alpha\theta) + c_2\sin(\alpha\theta).
    \]
\end{proof}

When we consider maps into $k$-pods, we recall that each edge of a $k$-pod is isometric to the half-line, $\hl$. Thus we must understand which homogeneous harmonic functions have image contained in the half-line.

\begin{proposition}\label{non-neg}
    A non-trivial homogeneous harmonic map $u(r,\theta)=r^\alpha \rho(\theta):S=S_{\theta_0}\to\hl$ exists if and only if $0<\alpha\theta_0\leq\pi$. Moreover when $0<\alpha\theta_0<\pi$, any pair of numbers $\rho(0),\rho(\theta_0)\geq 0$, not both 0, uniquely determine such a function.
\end{proposition}

\begin{proof}
    According to Lemma~\ref{triggy}, a homogeneous harmonic function $u:S_{\theta_0}\to\R$ has the form $u(r,\theta)=r^\alpha\rho(\theta)$ with
    \[
        \rho(r) = c_1\cos(\alpha\theta)+c_2\sin(\alpha\theta).
    \]
    If $\alpha\theta_0>\pi$, then both $\theta=0$ and $\theta=\pi$ are in the domain of $\rho$. If $c_2\neq0$ then either $\rho(1,0)$ or $\rho(1,\pi)$ is negative. And if $c_2=0$ but $c_1\neq0$ then $\rho$ changes signs around $\theta=\pi$. In this case no homogeneous harmonic map $u:S_{\theta_0}\to\hl$ exists.

    If $\alpha\theta_0=\pi$ then $\rho(\pi)=-\rho(0)$, so to be non-negative $\rho(0)=\rho(\pi)=0$. Now
    \[
        \rho(\theta) = c_2\sin(\alpha\theta).
    \]
    The values of $\sin(\alpha\theta)$ remain positive for $0<\theta<\theta_0$, so for the image of $u$ to lie in $\hl$ $c_2$ must be non-negative.
    
    Finally if $0<\alpha\theta_0<\pi$, Lemma~\ref{BVP} tells us that
    \[
        \rho(\theta) = \rho(0)\cos(\alpha\theta) + \frac{\rho(\theta_0)-\rho(0)\cos(\alpha\theta_0)}{\sin(\alpha\theta_0)}\sin(\alpha\theta).
    \]
    Clearly both $\rho(0)$ and $\rho(\theta_0)$ must be non-negative to ensure $u\geq0$. Now $\rho(\theta)$ only vanishes when
    \begin{eqnarray*}
        0 & = & \rho(0)\cos(\alpha\theta) + \frac{\rho(\theta_0)-\rho(0)\cos(\alpha\theta_0)}{\sin(\alpha\theta_0)}\sin(\alpha\theta)\\
        \rho(0)\cot(\alpha\theta) & = & \frac{\rho(0)\cos(\alpha\theta_0)-\rho(\theta_0)}{\sin(\alpha\theta_0)}\\
         & = & \rho(0)\cot(\alpha\theta_0) - \rho(\theta_0)\csc(\alpha\theta_0).
    \end{eqnarray*}
    If $\rho(0)=0$ then $\rho(\theta_0)=0$ also, in which case $u$ is a trivial map. And if $\rho(0)>0$ then $\rho(\theta)$ cannot vanish until $\cot(\alpha\theta)<\cot(\alpha\theta_0)$. But $\cot$ is a decreasing function, so this cannot happen for $0<\theta<\theta_0$.
    
    Thus in this last case where $0<\alpha\theta_0<\pi$, any pair of non-negative numbers $\rho(0)$ and $\rho(\theta_0)$, not both 0, uniquely determine a homogeneous harmonic function $u:S_{\theta_0}\to\hl$ of degree $\alpha$.
\end{proof}

Finally, when we discuss harmonic maps between 2-dimensional cones, we will need to understand harmonic maps between sectors. We will only consider those maps that send the boundary of the domain sector to the boundary of the target.

\begin{proposition}\label{maps between sectors}
    Fix $0<\theta_0,\varphi_0<\pi$. Then a non-trivial homogeneous harmonic map $u:S_{\theta_0}\to S_{\varphi_0}$ exists if and only if $0<\alpha\theta_0\leq\pi$. Moreover when $0<\alpha\theta_0<\pi$, any pair of numbers $\abs{u(1,0)},\abs{u(1,\theta_0)}\geq 0$, not both 0, uniquely determine such a map.
\end{proposition}

\begin{proof}
    Use polar coordinates $(r,\theta)$ in the domain sector $S_{\theta_0}$ and rectangular coordinates $(x,y)$ in the target sector $S_{\varphi_0}$. Then a homogeneous map $u:S_{\theta_0}$ can be written
    \[
        u(r,\theta) = \Big(x(r,\theta),y(r,\theta)\Big).
    \]
    As $S_{\varphi_0}\subset\R^2$ is flat, the map $u$ is harmonic if and only if the functions $x$ and $y$ are. So Lemma~\ref{triggy} says $x(r,\theta)=r^\alpha\chi(\theta)$ and $y(r,\theta)=r^\alpha\eta(\theta)$ with
    \begin{align*}
        \chi(\theta) & = c_1\cos(\alpha\theta)+c_2\sin(\alpha\theta)\\
        \eta(\theta) & = c_3\cos(\alpha\theta) + c_4\sin(\alpha\theta).
    \end{align*}
    
    First if $\alpha\theta_0>\pi$, then just as in Proposition~\ref{non-neg}, $\eta(\theta)$ cannot remain positive for $0\leq\theta\leq\theta_0$. But every point $(x,y)\in S_{\varphi_0}$ has $y\geq0$, so in this case no map $u:S_{\theta_0}\to S_{\varphi_0}$ exists.
    
    Now suppose $0<\alpha\theta_0\leq\pi$ and let $\rho_0=\abs{u(1,0)}$ and $\rho_1 = \abs{u(1,\theta_0)}$. In order to map the boundary of $S_{\theta_0}$ to the boundary of $S_{\varphi_0}$, the maps $\chi$ and $\eta$ must satisfy the following conditions:
    \begin{align*}
        \chi(0) & = \rho_0 & \chi(\theta_0) & = \rho_1\cos(\varphi_0)\\
        \eta(0) & = 0 & \eta(\theta_0) & = \rho_1\sin(\varphi_0).
    \end{align*}
    In other words,
    \begin{align*}
        c_1 & = \rho_0 & c_1\cos(\alpha\theta_0)+c_2\sin(\alpha\theta_0) & = \rho_1\cos(\varphi_0)\\
        c_3 & = 0 & c_3\cos(\alpha\theta_0)+c_4\sin(\alpha\theta_0) & = \rho_1\sin(\varphi_0).
    \end{align*}
    In all cases $c_1=\rho_0$ and $c_3=0$ are determined.
    
    If $\alpha\theta_0=\pi$, then
    \[
        -\rho_0 = \rho_1\cos(\varphi_0) \qquad\text{and}\qquad 0=\rho_1\sin(\varphi_0).
    \]
    Since $0<\varphi_0<\pi$, this means that $\rho_0=\rho_1=0$. Now
    \[
        \chi(\theta) = c_2\sin(\alpha\theta) \qquad\text{and}\qquad \eta(\theta)=c_4\sin(\alpha\theta).
    \]
    As long as $(c_2,c_4)\in S_{\varphi_0}$, the image of the corresponding map $u$ lies in $S_{\varphi_0}$.
    
    Finally if $0<\alpha\theta_0<\pi$ one can explicitly solve for the constants $c_2$ and $c_4$ to find
    \begin{align*}
        \chi(\theta) & = \rho_0\cos(\alpha\theta) + \frac{\rho_1\cos(\varphi_0)-\rho_0\cos(\alpha\theta_0)}{\sin(\alpha\theta_0)}\sin(\alpha\theta)\\
        \eta(\theta) & = \frac{\sin(\varphi_0)}{\sin(\alpha\theta_0)}\rho_1\sin(\alpha\theta).
    \end{align*}
    To verify that the corresponding map $u$ has image in $S_{\varphi_0}$, one needs only check that $\eta(\theta)\geq0$ and $\chi(\theta)\geq\eta(\theta)\cot(\varphi_0)$ for $0\leq\theta\leq\theta_0$. The first of these is immediate since the coefficient $\frac{\sin(\varphi_0)}{\sin(\alpha\theta_0)}\rho_1\geq0$ and $\sin(\alpha\theta)\geq0$ for $0\leq\theta\leq\theta_0$. The second condition is also straightforward.
    \begin{align*}
        \chi(\theta)-\eta(\theta)\cot(\varphi_0) & = \rho_0\cos(\alpha\theta) + \frac{\rho_1\cos(\varphi_0)-\rho_0\cos(\alpha\theta_0)}{\sin(\alpha\theta_0)}\sin(\alpha\theta) - \frac{\cos(\varphi_0)}{\sin(\alpha\theta_0)}\rho_1\sin(\alpha\theta)\\
         & = \rho_0\Big(\cos(\alpha\theta) - \cot(\alpha\theta_0)\sin(\alpha\theta)\Big)
    \end{align*}
    This expression only vanishes when $\cot(\alpha\theta)=\cot(\alpha\theta_0)$, which only happens when $\theta=\theta_0$ since $\cot(\alpha\theta)$ is decreasing on the interval $0<\theta<\theta_0$.
\end{proof}

\subsection{Balancing conditions}

We will mainly be interested in homogeneous maps in three situations.
\begin{enumerate}
    \item $f:C(\Gamma,\theta)\to\R^n$ from simplicial cones into Euclidean spaces,
    \item $f:C(C_n,\theta)\to C(k)$ from smooth cones into $k$-pods, and
    \item $f:C(\Gamma,\theta)\to C(\Gamma,\varphi)$ between 2-dimensional cones with the same simplicial structure.
\end{enumerate}
In addition to studying maps that are harmonic on the faces of the domain cone, we introduce balancing conditions along the edges of the domain, akin to those in e.g. \cite{daskal-mese-1,daskal-mese-2}, that have to do with harmonicity along the edges.

In all cases homogeneous maps send the vertex of one cone to the vertex of the other. In the third situation we will also demand that edges and faces of the cone are sent to themselves. We will then add extra conditions to ensure that our homogeneous maps are as smooth as the singularities of the domain and target allow. In particular, for open sets $U,V$ in top-dimensional faces of the domain and target respectively, if $f(U)\subset V$ then $f$ can be represented in coordinates as a smooth map between open subsets of Euclidean spaces. Along edges we introduce the so-called balancing conditions.

\subsubsection{Euclidean targets}

A map $f:C(\Gamma,\theta)\to\R^n$ can be written in coordinates as $f=(f_1,\ldots,f_n)$ for $n$ functions $f_j:C(\gamma,\theta)\to\R$. So we need only consider the case $n=1$, i.e. $f:C(\Gamma,\theta)\to\R$.

In this case the balancing conditions relate the normal derivatives in each face along a common edge of the domain.

\begin{definition}\label{euclidean balancing}
    Let $f:C(\Gamma,\theta)\to\R$ be a homogeneous function. Fix a vertex $v\in V(\Gamma)$ and let $\{e_j\}$ enumerate the edges of $\Gamma$ incident to $v$. In the corresponding faces take coordinates $\psi_j:C(e_j)\to S_{\theta(e_j)}$ adapted to v in order to represent $f$ in these faces by homogeneous functions
    \[
        u_j = f\circ\psi_j^{-1}:S_{\theta(e_j)}\to\R.
    \]
    The function $f$ is balanced along $C(v)$ if
    \[
        \sum_j \frac{\partial u_j}{\partial\theta}(r,0) = 0.
    \]
    The function $f$ is balanced if it is balanced along each edge of $C(\Gamma,\theta)$.
\end{definition}

For maps $f:C(\Gamma,\theta)\to\R^n$, simply impose the balancing condition for each component of $f$. We can see the regularity of $f$ along an edge $C(v)$ for a vertex of degree 2 by unfolding the incident faces $C(e_1)$ and $C(e_2)$ to form a larger sector of the plane.

\begin{lemma}\label{euclidean smooth}
    Let $f:(\Gamma,\theta)\to\R$ be a balanced homogeneous function and let $v\in V(\Gamma)$ be a vertex of degree 2. Then $f$ is $C^1$ over the edge $C(v)\subset C(\Gamma,\theta)$.
\end{lemma}

\begin{proof}
    Let $e_1,e_2\in E(\Gamma)$ be the two edges incident to $v$. In the corresponding faces take coordinates $\psi_j:C(e_j)\to S_{\theta(e_j)}$ adapted to $v$ and represent $f$ in these coordinates by homogeneous functionss
    \[
        u_j = f\circ\psi_j^{-1}:S_{\theta(e_j)}\to\R.
    \]
    
    Now define a function $u:S=\{(r,\theta)\in\R^2\vert r\geq0,-\theta(e_2)\leq\theta\leq\theta(e_1)\}\to\R$ as follows.
    \[
        u(r,\theta) = \begin{cases}
            u_1(r,\theta), & \theta\geq0\\
            u_2(r,-\theta), & \theta\leq0.
        \end{cases}
    \]
    The maps $u_1$ and $u_2$ agree on the $x$-axis $\{(r,0)\vert r\geq0\}$. Thus $u$ is continuous, and moreover $\frac{\partial u_1}{\partial r}(r,0)=\frac{\partial u_2}{\partial r}(r,0)$ so $\frac{\partial u}{\partial r}$ is continuous. The $\frac{\partial}{\partial\theta}$ derivatives of $u$ are given by
    \[
        \frac{\partial u}{\partial\theta}(r,\theta) = \begin{cases}
            \frac{\partial u_1}{\partial\theta}(r,\theta), & \theta\geq0\\
            -\frac{\partial u}{\partial\theta}(r,-\theta), & \theta\leq0.
        \end{cases}
    \]
    At $\theta=0$ the balancing condition says that these partial derivatives coincide, so $\frac{\partial u}{\partial\theta}$ is continuous too.
\end{proof}

\subsubsection{1-dimensional targets}\label{S 1-d balancing}

Note that for a homogeneous map $f=(y(x),s(x)t^\alpha):C(C_n,\theta)\to C(k)$ into a $k$-pod, in order to be continuous the map $y$ must be locally constant where $s\neq0$. We will also assume that if $s(v)=0$ at some vertex $v$ with incident edges $e_1$ and $e_2$, then $y(e_1)\neq y(e_2)$.

Let $f=(y(x),s(x)t^\alpha):C(C_n,\theta)\to C(k)$ be a homogeneous map from a smooth cone to a $k$-pod. If $s(x)=0$ for some point $x$ in the interior of an edge $e\in E(\Gamma)$, subdivide $\Gamma$ by introducing a vertex at $x$. After performing all such subdivisions, we'll abuse notation and call the resulting space $C(C_n,\theta)$ too. Now $s(x)$ vanishes only at the vertices of $\Gamma$, and $f$ maps each face $C(e)$ to a single edge of $C(k)$.

Now for a homogeneous map $f:C(C_n,\theta)\to C(k)$ we can define the balancing condition in terms of normal derivatives along each edge of the domain.

\begin{definition}\label{1-d balancing}
    Let $f:C(C_n,\theta)\to C(k)$ be a homogeneous map. Fix a vertex $v\in V(C_n)$ and let $e_1,e_2\in E(C_n)$ be the two edges incident to $v$. In the corresponding faces take coordinates $\psi_j:C(e_j)\to S_{\theta(e_j)}$ adapted to $v$, in order to represent $f$ in these faces by homogeneous maps
    \[
        u_j = f\circ\psi_j^{-1}:S_{\theta(e_j)}\to\hl.
    \]

    If $f(v,1)$ is the vertex of $C(k)$ then $f$ is balanced along $C(v)$ if
    \[
        \frac{\partial u_1}{\partial\theta}(r,0) = \frac{\partial u_2}{\partial \theta}(r,0).
    \]
    If $f(v,1)$ lies in the interior of some edge of $C(k)$ then $f$ is balanced along $C(v)$ if
    \[
        \frac{\partial u_1}{\partial\theta}(r,0) + \frac{\partial u_2}{\partial\theta}(r,0) = 0.
    \]
    The map $f$ is balanced if it is balanced along each edge of $C(C_n,\theta)$.
\end{definition}

In this case we can again see the regularity of $f$ along each edge $C(v)\subset C(C_n,\theta)$. In addition to unfolding the two faces incident to $C(v)$ into a larger sector, we may have to unfold their image edges in $C(k)$ to form a whole line if $f(v)$ is the vertex of $C(k)$.

\begin{lemma}\label{k-pod smooth}
    Let $f=(y(x),s(x)t^\alpha):(C_n,\theta)\to C9k)$ be a balanced homogeneous map and $v\in V(\Gamma)$. Then $f$ is $C^1$ over the edge $C(v)\subset C(C_n,\theta)$.
\end{lemma}

\begin{proof}
    Let $e_1,e_2\in E(\Gamma)$ be the two edges incident to $v$. In the corresponding faces take coordinates $\psi_j:C(e_j)\to S_{\theta(e_j)}$ adapted to $v$.

    Consider first the case where $s(v)\neq0$, so $f(v,1)$ is in the interior of an edge of $C(k)$. The map $f$ is represented in coordinates in the faces $C(e_1)$ and $C(e_2)$ by homogeneous maps into the same half line $\hl\subset\R$. Thus Lemma~\ref{euclidean smooth} already establishes the $C^1$ regularity of $f$ over the edge $C(v)$.
    
    Now consider $s(v)=0$, so $f(v,1)$ is the vertex of $C(k)$. Again we can represent the map $f$ in the face $C(e_j)$ by homogeneous maps
    \[
        u_j = f\circ\psi_j^{-1}:S_{\theta(e_j)}\to\hl.
    \]
    Since $y(e_1)\neq y(e_2)$, the target half-lines should be thought of as different spaces, and can thus be glued together to form the entire real line. Specifically, define $u:S=\{(r,\theta)\in\R^2\vert r\geq0,-\theta(e_2)\leq\theta\leq\theta(e_1)\}\to\R$ as follows.
    \[
        u(r,\theta) = \begin{cases}
            u_1(r,\theta), & \theta\geq0\\
            -u_2(r,-\theta), & \theta\leq0.
        \end{cases}
    \]

    Both $u_1$ and $u_2$ map $(r,0)$ to $0$, so $u$ and $\frac{\partial u}{\partial r}$ are continuous over the $x$-axis. The $\frac{\partial}{\partial\theta}$ derivatives are given by
    \[
        \frac{\partial u}{\partial\theta}(r,\theta) = \begin{cases}
            \frac{\partial u_1}{\partial\theta}(r,\theta), & \theta\geq0\\
            \frac{\partial u_2}{\partial\theta}(r,-\theta), & \theta\leq 0.
        \end{cases}
    \]
    Again the balancing condition says that these derivatives coincide at $(r,0)$, so $\frac{\partial u}{\partial\theta}$ is continuous.
\end{proof}

\subsubsection{2-dimensional targets}

Let $f=(y(x),s(x)t^\alpha):C(\Gamma_1,\theta_1)\to C(\Gamma_2,\theta_2)$ be a homogeneous map. If $s(x)\in V(\Gamma_2)$ for some point $x$ in the interior of an edge $e\in E(\Gamma_1)$, subdivide $\Gamma_1$ by introducing a vertex at $x$. Likewise subdivide $\Gamma_2$ at any point $s(v)$ for $v\in V(\Gamma_1)$. After performing all such subdivisions, the original map factors as a composition $C(\Gamma_1,\theta_1)\to C(\Gamma_1,\varphi)\to C(\Gamma_2,\theta_2)$ where the first map sends each edge and face of $C(\Gamma_1)$ to itself and the second map sends each face of $C(\Gamma_1,\varphi)$ isometrically to a face of $C(\Gamma_2,\theta_2)$.

As the harmonicity of a map is unaffected by postcomposition with isometries, we may consider only maps $f(x,t)=(y(x),s(x)t^\alpha):C(\Gamma,\theta)\to C(\Gamma,\varphi$ such that $y(v)=v$ for all $v\in V(\Gamma)$ and $y\vert_e$ maps $e$ to $e$ for each $e\in E(\Gamma)$. That is, $f$ maps each edge and face of $C(\Gamma)$ to itself.

Now for a homogeneous map $f:C(\Gamma,\theta)\to C(\Gamma,\varphi)$ we can define the balancing condition in terms of normal derivatives along each edge of the domain. But under the additional assumptions made on our maps, only some components of the normal derivatives are relevant.

\begin{definition}\label{2-d balancing}
    Let $f:C(\Gamma,\theta)\to C(\Gamma,\varphi)$ be a homogeneous map. Fix a vertex $v\in V(\Gamma_1)$ and let $\{e_j\}$ enumerate the edges of $\Gamma$ incident to $v$. In each face $C(e_j)$ of $C(\Gamma,\theta)$ take coordinates adapted to $v$, and likewise in the faces of $C(\Gamma,\varphi)$. In this way, represent $f$ in each face by a homogeneous map
    \[
        u_j:S_j=S_{\theta(e_j)}\to S_{\varphi(e_j)}=S'_j.
    \]
    Taking rectangular coordinates on the target $S'_j$, write $u_j = (x_j,y_j)$. Then the map $f$ is balanced along $C(v)$ if
    \[
        \sum_j\frac{\partial x_j}{\partial\theta}(r,0) = 0.
    \]
    The map $f$ is balanced if it is balanced along each edge of $C(\Gamma,\theta)$.
\end{definition}

From the maps $x_j$ in this definition one can define a map $x$ as follows. Suppose there are $N$ edges $e_j$ incident to $v$, and fix some index $1\leq j_0\leq N$. Define
\[
    x(r,\theta) = \begin{cases}
        x_{j_0}(r,\theta), & \theta\geq0\\
        -x_{j_0}(r,-\theta)+\frac{2}{N}\sum_jx_j(r,-\theta), & \theta\leq0.
    \end{cases}
\]
The fact that all the functions $x_j$ agree along $\theta=0$ means that $x$ and $\frac{\partial x}{\partial r}$ are continuous, and the balancing condition now means that $\frac{\partial x}{\partial\theta}$ is continuous.

\begin{remark}
    In the case that $v$ has degree 2, the function $x$ constructed just above is the direct analogue of the function $u$ from Lemma~\ref{euclidean smooth}. Unfortunately there is not an analogous smoothness result for a similarly constructed $y$ function. The restriction that $f$ sends $C(v)$ to itself for all $v\in V(\Gamma)$ means that we can not expect such a smoothness result.
\end{remark}

\section{Balanced harmonic functions into Euclidean spaces}\label{S euclidean target}

A function $f:C(\Gamma,\theta)\to\R$ is harmonic if its restriction to each face of $C(\Gamma,\theta)$ is a harmonic function on a plane sector. In this section we investigate the balanced (as in Definition~\ref{euclidean balancing}) homogeneous harmonic functions $f:C(\Gamma,\theta)\to\R$. The analysis will be different if $\alpha\theta(e)\in\pi\mathbb{Z}$ for some edge $e\in E(\Gamma)$; we will refer to such $\alpha$ as \emph{singular degrees}.

We will focus first on the case of non-singular $\alpha$. In each case we begin by constructing homogeneous harmonic functions $f:C(\Gamma,\theta)\to\R$, and then investigating which of those functions are balanced.

\subsection{Non-singular degrees}

\begin{lemma}\label{rho determines f}
    If $\alpha$ is a non-singular degree then a homogeneous harmonic function $f:C(\Gamma,\theta)\to\R$ of degree $\alpha>0$ is uniquely determined by a function $\rho:V(\Gamma)\to\R$.
\end{lemma}

\begin{proof}
    For a fixed edge $e\in E(\Gamma)$ consider the face $C(e)\subset C(\Gamma,\theta)$, which is isometric to $S_{\theta(e)}\subset\R^2$. According to Lemma~\ref{BVP} a homogeneous harmonic function $u:S_{\theta(e)}\to\R$ is uniquely determined by the values $u(1,0)$ and $u(1,\theta_0)$ (using polar coordinates $(r,\theta)$ in $S_{\theta(e)}$.
    
    Fix a function $\rho:V(\Gamma)\to\R$. For each edge $e=(v_0v_1)\in E(\Gamma)$, fix coordinates in the face $C(e)$ adapted to $v_0$, so that $C(e)$ is identified with $S_{\theta(e)}$ and the edge $C(v)$ is identified with the positive $x$-axis. Define $f$ in $C(e)$ to be represented in these coordinates by the unique homogeneous harmonic function $u:S_{\theta(e)}\to\R$ with $u(1,0)=\rho(v_0)$ and $u(1,\theta(e))=\rho(v_1)$.
    
    Along each edge $C(v)$, $f(v,t)=\rho(v)t^\alpha$ is defined independently of the choice of incident face. Hence $f$ is continuous over the edges of $C(\Gamma,\theta)$ and harmonic in each face.
\end{proof}

For ease of notation, we will enumerate the vertices of $\Gamma$ by $\{v_i\}$. Then a function $\rho:V(\Gamma)\to\R$ can be encoded by the vector $(\rho_i=\rho(v_i))$. If $v_iv_j\in E(\Gamma)$ is an edge, we will also denote by $\theta_{ij}=\theta(v_iv_j)$.

\begin{theorem}\label{euclidean balanced}
    The homogeneous harmonic function $f:C(\Gamma,\theta)\to\R$ of non-singular degree $\alpha>0$ determined by $\rho:V(\Gamma)\to\R$ is balanced if and only if for each vertex $v_i$,
    \[
        \sum_{v_j\sim v_i} \frac{\rho_j-\cos(\alpha\theta_{ij})\rho_i}{\sin(\alpha\theta_{ij})}=0.
    \]
\end{theorem}
Here $v\sim w$ means that vertices $v$ and $w$ are adjacent, so there is an edge $vw\in E(\Gamma)$.

\begin{proof}
    In each face $C(v_iv_j)$ choose coordinates adapted to $v_i$, so that the face is identified with $S_{\theta_{ij}}$ and $C(v_i)$ is identified with the positive $x$-axis. According to Lemma~\ref{BVP}, in these coordinates $f\vert_{C(v_iv_j)}$ is represented by
    \[
        u_j(r,\theta) = r^\alpha\left(\rho_i\cos(\alpha\theta) + \frac{\rho_j-\cos(\alpha\theta_{ij})\rho_i}{\sin(\alpha\theta_{ij})}\sin(\alpha\theta)\right).
    \]
    And the balancing formula of Definition~\ref{euclidean balancing} says that $f$ is balanced along $C(v_i)$ if
    \[
        \sum_{v_j\sim v_i} \frac{\partial u_j}{\partial\theta}(r,0) = 0.
    \]
    
    Combining these two formulas, we see that $f$ is balanced along $C(v_i)$ if
    \[
        \alpha r^\alpha\sum_{v_j\sim v_i} \frac{\rho_j-\cos(\alpha\theta_{ij})\rho_i}{\sin(\alpha\theta_{ij}))} = 0.
    \]
    In order to be balanced, $f$ must be balanced along each edge of $C(\Gamma,\theta)$, so every sum of this form must vanish.
\end{proof}

Another way to interpret this result is to say a vector $(\rho_j)$ determines a balanced homogeneous harmonic function $f:C(\Gamma,\theta)\to\R$ of non-singular degree $\alpha>0$ if it is in the kernel of the matrix $\Delta_{\alpha\theta}$ with entries
\[
    (\Delta_{\alpha\theta})_{ij} = \begin{cases}
        -\sum_{v_k\sim v_i}\cot(\alpha\theta_{ik}), & i=j\\
        \csc(\alpha\theta_{ij}), & v_i\sim v_j\\
        0 & else.
    \end{cases}
\]
For only certain values of $\alpha$ will $\Delta_{\alpha\theta}$ have a non-trivial kernel. That is, only certain degrees $\alpha$ admit a balanced homogeneous harmonic function.

\begin{remark}
    In general one can find a balanced homogeneous harmonic map $f:C(\Gamma,\theta)\to\R^d$ for any $d$ ($f$ may be the 0 map). Both the harmonicity and the balancing condition for $f$ split into the corresponding conditions for the components $f_j$ of $f=(f_1,\ldots,f_d)$. By setting the components $\{f_j\}$ to be a basis of the $d$-dimensional kernel of $\Delta_{\alpha\theta}$ one finds the most geometrically interesting such maps. If one tries to map into a larger dimensional space, or more generally lets the components of $f$ be dependent, then the image of $f$ lies in a subspace, not taking advantage of the full dimension of the target.
\end{remark}

In special cases the collection of degrees $\alpha$ that admit balanced homogeneous harmonic functions is related to other discrete Laplace operators studied extensively in the literature. The following result is a direct generalization of Proposition 13 from \cite{daskal-mese-2}.

\begin{proposition}\label{constant theta}
    Suppose $\theta(e)=\theta_0$ for all $e\in E(\Gamma)$. If $\alpha\theta_0\not\in\pi\mathbb{Z}$ then there exist balanced harmonic functions of degree $\alpha$ if and only if $1-\cos(\alpha\theta_0)$ is an eigenvalue of the normalized graph Laplacian $\mathcal{L}$ on $\Gamma$.
\end{proposition}

\begin{remark}
    This Proposition ignores the eigenvalue $\lambda=0$ that $\mathcal{L}$ necessarily has, and the eigenvalue $\lambda=2$ that $\mathcal{L}$ may have. These cases correspond to $\alpha\theta_0 = 2n\pi$ and $\alpha\theta_0=(2n+1)\pi$, respectively, and will be discussed in Section~\ref{S euclidean singular}.
\end{remark}

\begin{proof}
    Suppose $\theta(e)=\theta_0$ for all $e$ and that $(\rho_i)$ is in the kernel of $\Delta_{\alpha\theta}$. For a fixed vertex $v_i\in V(\Gamma)$ with $deg(v_i)$ neighbors,
    \begin{align*}
        0 & = \sum_{v_j\sim v_i}\Big(\csc(\alpha\theta_0)\rho_j-\cot(\alpha\theta_0)\rho_i\Big)\\
         & = \csc(\alpha\theta_0)\sum_{v_j\sim v_i}\Big(\rho_j-\rho_i\Big) + deg(v_i)\Big(\csc(\alpha\theta_0)-\cot(\alpha\theta_0)\Big)\rho_i.\\
        \Big(1-\cos(\alpha\theta_0)\Big)\rho_i & = \frac{1}{deg(v_i)}\sum_{v_j\sim v_i}\Big(\rho_i-\rho_j\Big).
    \end{align*}
    
    Thus $1-\cos(\alpha\theta_0)$ is an eigenvalue of the vertex-weighted Laplace matrix $L$ with entries
    \[
        L_{ij} = \begin{cases}
            1, & i=j\\
            \frac{-1}{deg(v_i)}, & v_i\sim v_j\\
            0, & v_i\not\sim v_j.
        \end{cases}
    \]
    Although this matrix $L$ is not symmetric, it is similar to the symmetric normalized Laplace matrix $\mathcal{L}$ (see \cite{chung-langlands} or \cite{banerjee-jost}) with entries
    \[
        \mathcal{L}_{ij} = \begin{cases}
            1, & i=j\\
            \frac{-1}{\sqrt{deg(v_i)deg(v_j)}}, & v_i\sim v_j\\
            0, & v_i\not\sim v_j.
        \end{cases}
    \]
    Since these matrices are similar they have the same eigenvalues. Their spectra have been studied extensively in the literature, for instance in \cite{chung} and \cite{banerjee-jost}.
\end{proof}

For metrics given by general functions $\theta:E(\Gamma)\to(0,\pi)$, the question of which degrees $\alpha$ admit balanced homogeneous harmonic functions is more subtle. We will study the structure of $\Delta_{\alpha\theta}$ via the quadratic form it determines on $\R^{\#V(\Gamma)}$. This quadratic form will give some information about the changing eigenvalues of $\Delta_{\alpha\theta}$. In particular, when an eigenvalue crosses 0 the operator $\Delta_{\alpha\theta}$ will have a non-trivial kernel.

\begin{lemma}\label{increasing}
    For each fixed $\rho=(\rho_i)\not\equiv0$ the quadratic form $\rho\cdot\Delta_{\alpha\theta}\rho$ is strictly increasing in $\alpha$.
\end{lemma}

\begin{proof}
    First compute the quadratic form as follows.
    \begin{align*}
        \rho\cdot\Delta_{\alpha\theta}\rho & = \sum_i \rho_i\left(\sum_{v_j\sim v_i}\frac{\rho_j-\cos(\alpha\theta_{ij})\rho_i}{\sin(\alpha\theta_{ij})}\right)\\
         & = \sum_{v_iv_j\in E(\Gamma)}\left(\rho_i\frac{\rho_j-\rho_i\cos(\alpha\theta_{ij})}{\sin(\alpha\theta_{ij})} + \rho_j\frac{\rho_i-\rho_j\cos(\alpha\theta_{ij})}{\sin(\alpha\theta_{ij})}\right)\\
         & = \sum_{v_iv_j\in E(\Gamma)}\frac{2\rho_i\rho_j-\cos(\alpha\theta_{ij})\Big(\rho_i^2+\rho_j^2\Big)}{\sin(\alpha\theta_{ij})}\\
         & = \sum_{v_iv_j\in E(\Gamma)}\Big(2\csc(\alpha\theta_{ij})\rho_i\rho_j - \cot(\alpha\theta_{ij})\big(\rho_i^2+\rho_j^2\big)\Big).
    \end{align*}
    Now the derivative with respect to $\alpha$ is easy to compute.
    \begin{align*}
        \frac{\partial}{\partial\alpha}\Big(\rho\cdot\Delta_{\alpha\theta}\rho\Big) & = \sum_{v_iv_j\in E(\Gamma)}\Big(\theta_{ij}\csc^2(\alpha\theta_{ij})\big(\rho_i^2+\rho_j^2\big) - 2\rho_i\rho_j\theta_{ij}\csc(\alpha\theta_{ij})\cot(\alpha\theta_{ij})\Big)\\
         & = \sum_{v_iv_j\in E(\Gamma)}\theta_{ij}\csc^2(\alpha\theta_{ij})\Big(\rho_i^2+\rho_j^2 - 2\rho_i\rho_j\cos(\alpha\theta_{ij})\Big)\\
         & \geq \sum_{v_iv_j\in E(\Gamma)}\theta_{ij}\csc^2(\alpha\theta_{ij})\Big(\rho_i-\rho_j\Big)^2 \geq 0.
    \end{align*}
    
    Moreover suppose $\frac{\partial}{\partial\alpha}\Big(\rho\cdot\Delta_{\alpha\theta}\rho\Big)=0$ at some fixed $\rho=(\rho_i)$ and some fixed $\alpha$. Then $\rho_i\rho_j\cos(\alpha\theta_{ij}) = \rho_i\rho_j$ for each $v_iv_j\in E(\Gamma)$, and also
    \[
        \sum_{v_iv_j\in E(\Gamma)}\theta_{ij}\csc^2(\alpha\theta_{ij})\Big(\rho_i-\rho_j\Big)^2 = 0.
    \]
    The expression on the left is clearly non-negative, so for equality to hold $\rho$ must be constant. And since $\alpha$ is nonsingular, $\cos(\alpha\theta_{ij})\neq1$ so that $\rho_i\rho_j=0$ for all $v_iv_j\in E(\Gamma)$. Since $\rho$ is constant, we have $\rho\equiv0$. That is, the only way that $\frac{\partial}{\partial\alpha}\Big(\rho\cdot\Delta_{\alpha\theta}\rho\Big)=0$ at any $\alpha$ is if $\rho\equiv 0$.
\end{proof}

\begin{theorem}\label{eigenvalues}
    In an interval of non-singular $\alpha$, the eigenvalues of $\Delta_{\alpha\theta}$ are strictly increasing functions of $\alpha$.
\end{theorem}

\begin{proof}
    The entries of the matrix $\Delta_{\alpha\theta}$ are analytic in $\alpha$ except where they are not defined, i.e. at singular $\alpha$. Since we avoid such values, $\Delta_{\alpha\theta}$ is analytic in $\alpha$. Moreover, $\Delta_{\alpha\theta}$ is symmetric for each real $\alpha$. A result in the perturbation theory of eigenvalue problems (see \cite{rellich-berkowitz} Section 1.1 or \cite{kato} Theorem 6.1) says that the eigenvalues and eigenvectors of $\Delta_{\alpha\theta}$ are also analytic in $\alpha$. Namely, there are $\lambda_j(\alpha)\in\R$ and $w_j(\alpha)\in\R^{\#V(\Gamma)}$ depending analytically on $\alpha$ so that
    \[
        \Delta_{\alpha\theta}w_j(\alpha) = \lambda_j(\alpha)w_j(\alpha).
    \]
    
    Without loss of generality, $w_j(\alpha)$ is a unit vector for each $j$ and each $\alpha$. The eigenvalues can be recovered from the eigenvectors via the formula
    \[
        \lambda_j(\alpha) = w_j(\alpha)\cdot\Delta_{\alpha\theta}w_j(\alpha).
    \]
    Differentiating this identity with respect to $\alpha$ yields
    \[
        \lambda'_j(\alpha) = 2w'_j(\alpha)\cdot\Delta_{\alpha\theta}w_j(\alpha) + w_j(\alpha)\cdot\frac{\partial\Delta_{\alpha\theta}}{\partial\alpha}w_j(\alpha).
    \]
    
    As each $w_j(\alpha)$ is a unit vector, it lies in the unit sphere of $\R^{\#V(\Gamma)}$. Hence $w'(\alpha)$ is perpendicular to $w_j(\alpha)$. But $\Delta_{\alpha\theta}w_j(\alpha) = \lambda_j(\alpha)w_j(\alpha)$ is parallel to $w_j(\alpha)$, so the first term in $\lambda'(\alpha)$ vanishes. And the second term is strictly positive by Lemma~\ref{increasing}. Thus
    \[
        \lambda'_j(\alpha)>0.
    \]
\end{proof}

An immediate consequence of this Theorem is that one can count the number of degrees $\alpha$ in an interval that admit balanced homogeneous harmonic functions in terms of the behavior of $\rho\cdot\Delta_{\alpha\theta}\rho$ at the endpoints of the interval.

\begin{corollary}
    Suppose $[\alpha_0,\alpha_1]$ is an interval of non-singular degrees. If $\Delta_{\alpha_0\theta}$ has $n_0$ non-positive eigenvalues and $\Delta_{\alpha_1\theta}$ has $n_1$ negative eigenvalues then there are $n_0-n_1$ degrees $\alpha\in[\alpha_0,\alpha_1]$, counted with multiplicity, for which $\Delta_{\alpha\theta}$ has a non-trivial kernel.
\end{corollary}

\begin{remark}
    If $\Delta_{\alpha\theta}$ has a non-trivial kernel, the \emph{multiplicity} of $\alpha$ is the dimension of that kernel. This is the number of independent balanced homogeneous harmonic functions $f:C(\Gamma,\theta)\to\R$ of degree $\alpha$, or equivalently the maximum dimension of the image of a balanced homogeneous harmonic map $f:C(\Gamma,\theta)\to\R^d$.
\end{remark}

\begin{proof}
    Let the eigenvalues of $\Delta_{\alpha\theta}$ be $\{\lambda_j(\alpha)\}$ as in the proof of Theorem~\ref{eigenvalues}. The operator $\Delta_{\alpha\theta}$ can only have a non-trivial kernel if some $\lambda_j(\alpha)=0$.
    
    According to Theorem~\ref{eigenvalues} each $\lambda_j(\alpha)$ is strictly increasing in $\alpha$. Let $\lambda_1(\alpha_0),\ldots,\lambda_{n_0}(\alpha_0)$ be the non-positive eigenvalues of $\Delta_{\alpha_0\theta}$. For each of these eigenvalues there is at most one $\alpha\in[\alpha_0,\alpha_1]$ for which $\lambda_j(\alpha)=0$; before this degree $\lambda_j<0$ and after it $\lambda_j>0$. The remaining eigenvalues, $\{\lambda_j(\alpha)\vert j>n_0\}$, all remain positive for $\alpha_0\leq\alpha\leq\alpha_1$.
    
    So if there are $n_1$ negative eigenvalues of $\Delta_{\alpha_1\theta}$ then $n_0-n_1$ of the original $n_0$ non-positive eigenvalues of $\Delta_{\alpha_0\theta}$ were 0 for some $\alpha\in[\alpha_0,\alpha_1]$.
\end{proof}

One can extend this result to open intervals whose endpoints are singular degrees. If $\alpha_0\theta_{ij}=n\pi$ for some edge $v_iv_j\in E_(\Gamma)$ and some $n\in\mathbb{Z}$, the term in the quadratic form $\rho\cdot\Delta_{(\alpha_0+t)\theta}\rho$ corresponding to $v_iv_j$ reads
\begin{align*}
    \frac{2\rho_i\rho_j - \cos\big((\alpha_0+t)\theta_{ij}\big)\Big(\rho_i^2+\rho_j^2\Big)}{\sin\big((\alpha_0+t)\theta_{ij}\big)} & = \frac{2\rho_i\rho_j-\big((-1)^n+O(t^2)\big)\Big(\rho_i^2+\rho_j^2\Big)}{(-1)^n\theta_{ij}t+O(t^3)}\\
     & = -\frac{\Big(\rho_i-(-1)^n\rho_j\Big)^2+O(t^2)}{\theta_{ij}t+O(t^3)}.
\end{align*}
If $\rho_i-(-1)^n\rho_j\neq0$ then this term approaches $\infty$ as $t\to0$ from the left, and approaches $-\infty$ as $t\to0$ from the right. But if $\rho_i-(-1)^n\rho_j=0$ then this term approaches $0$ as $t\to0$.

Suppose $\alpha_0\theta(e)\in\pi\mathbb{Z}$ for more than one edge $e\in E(\Gamma)$. For $\rho\cdot\Delta_{\alpha\theta}\rho$ to remain bounded near $\alpha_0$, each edge $v_iv_j$ with $\alpha_0\theta_{ij}=n_{ij}\pi$ imposes a relation $\rho_i-(-1)^{n_{ij}}\rho_j=0$. These relations cut out a subspace $B(\alpha_0)\subset\R^{\#V(\Gamma)}$ on which $\rho\cdot\Delta_{\alpha\theta}\rho$ remains bounded for $\alpha$ near $\alpha_0$. But for any $\rho\not\in B(\alpha_0)$ the value $\rho\cdot\Delta_{\alpha\theta}\rho$ satisfies
\[
    \lim_{\alpha\to\alpha_0^{\pm}} \rho\cdot\Delta_{\alpha\theta}\rho = \mp\infty.
\]

Now suppose $\alpha_0<\alpha_1$ are singular, but each $\alpha\in(\alpha_0,\alpha_1)$ is non-singular. Let $B_j=B(\alpha_j)$ be the subspaces described above, and let $Q_j(\rho) = \lim_{\alpha\to\alpha_j}\rho\cdot\Delta_{\alpha\theta}\rho$ be a quadratic form defined for $\rho\in B_j$. Then the number of non-positive eigenvalues of $\Delta_{\alpha\theta}$ for $\alpha=\alpha_0+\epsilon$, $\epsilon$ sufficiently small, is given by the number of negative eigenvalues of $Q_0$ plus $dim(B_0^\perp)$. Likewise the number of positive eigenvalues of $\Delta_{\alpha\theta}$ for $\alpha=\alpha_1-\epsilon$ is given by the number of positive eigenvalues of $Q_1$ plus $dim(B_1^\perp)$. The number of degrees $\alpha\in (\alpha_0,\alpha_1)$ admitting a balanced homogeneous harmonic function, counted with multiplicity, is the difference between these two numbers.

The behavior of $\Delta_{\alpha\theta}$ as $\alpha\to0$ is of particular interest.

\begin{proposition}\label{degree 0}
    As $\alpha\to0$ from the right, the limit of the operators $\alpha\Delta_{\alpha\theta}$ converge to an edge weighted Laplace operator on $\Gamma$ with edge weights $w(e) = \frac{1}{\theta(e)}$. In particular $B(0)$ is the span of the constant vector $\mathbb{1}$, with $\mathbb{1}_i=1$ for all $i$, and $Q_0 = 0$ on $B_0$.
\end{proposition}

\begin{proof}
    The limit of the quadratic forms determined by $\alpha\Delta_{\alpha\theta}$ is given by
    \begin{align*}
        \lim_{\alpha\to0^+}\rho\cdot\Big(\alpha\Delta_{\alpha\theta}\Big)\rho & = \lim_{\alpha\to0^+}\sum_{v_iv_j\in E(\Gamma)}\frac{2\rho_i\rho_j-\cos(\alpha\theta_{ij})\Big(\rho_i^2+\rho_j^2\Big)}{\sin(\alpha\theta_{ij})/\alpha}\\
         & = \sum_{v_iv_j\in E(\Gamma)}\frac{2\rho_i\rho_j - \rho_i^2-\rho_j^2}{\theta_{ij}}\\
         & = -\sum_{v_iv_j\in E(\Gamma)}\frac{(\rho_i-\rho_j)^2}{\theta_{ij}}\\
         & = \rho\cdot\Delta\rho,
    \end{align*}
    where the operator $\Delta$ has entries
    \[
        \Delta_{ij} = \begin{cases}
            \sum_{v_k\sim v_i}\frac{1}{\theta_{ik}}, & i=j\\
            \frac{1}{\theta_{ij}}, & v_i\sim v_j\\
            0 & else.
        \end{cases}
    \]
    This is precisely the edge-weighted Laplace operator described in the statement of the Proposition.
    
    Now if  $\rho_i=c$ for all $i$ then
    \[
        \lim_{\alpha\to0^+}\rho\cdot\Delta_{\alpha\theta}\rho = 0.
    \]
    But for any other $\rho$,
    \[
        \lim_{\alpha\to0^+}\rho\cdot\Big(\alpha\Delta_{\alpha\theta}\Big)\rho = -\sum_{v_iv_j\in E(\Gamma)}\frac{(\rho_i-\rho_j)^2}{\theta_{ij}} < 0,
    \]
    so $\rho\cdot\Delta_{\alpha\theta}\rho\to-\infty$ as $\alpha\to0$ from the right.
\end{proof}

We are now in a position to give a lower bound on the smallest degree $\alpha>0$ of any balanced homogeneous harmonic function $f:C(\Gamma,\theta)\to\R$, in terms of the first eigenvalue of the normalized Laplace operator $\mathcal{L}$ with entries
\[
    \mathcal{L}_{ij} = \begin{cases}
        1, & i=j\\
        -\frac{1}{\sqrt{deg(v_i)deg(v_j)}}, & v_i\sim v_j\\
        0 & v_i\not\sim v_j.
    \end{cases}
\]

\begin{theorem}
    Let $\lambda_1$ be the first positive eigenvalue of $\mathcal{L}$, and let
    \[
        \theta_{\max} = \max_{e\in E(\Gamma)}\theta(e).
    \]
    If $\Gamma$ is not a complete graph and there is a balanced homogeneous harmonic function $f:C(\Gamma,\theta)\to\R$ of degree $\alpha$, then
    \[
        \alpha\geq\frac{1}{\theta_{\max}}\arccos(1-\lambda_1).
    \]
    If $\Gamma=K_n$ is the complete graph on $n$ vertices then
    \[
        \alpha\geq\frac{2}{\theta_{\max}}\arctan\left(\frac{n}{n-1}\right).
    \]
\end{theorem}

\begin{proof}
    The operator $\mathcal{L}$ is positive semidefinite, which can be seen by the formula
    \[
        w\cdot\mathcal{L}w = \sum_{v_iv_j\in E(\Gamma)}\left(\frac{w_i}{\sqrt{deg(v_i)}} - \frac{w_j}{\sqrt{deg(v_j)}}\right)^2.
    \]
    Moreover, $\mathcal{L}$ has a kernel spanned by the vector $w=D^{1/2}\mathbb{1}$ with $w_i = \sqrt{deg(v_i)}$. If $\Gamma$ is not complete then $\lambda_1\leq1$, and if $\Gamma=K_n$ then $\lambda_1=\frac{n}{n-1}$ (see \cite{chung} Lemma 1.7).

    We will proceed by contradiction. In the case where $\Gamma$ is not complete and $\lambda_1\leq1$, we first assume that $\alpha<\arccos(1-\lambda_1)/\theta_{\max}$. The bound $\lambda_1\leq1$ also implies
    \[
        \alpha\theta_{\max} < \frac{\pi}{2}.
    \]
    Since the first singular $\alpha$ happens when $\alpha\theta_{\max} = \pi$, we are within the interval of non-singular degrees with $0$ as an endpoint.
    
    Using the bound $\alpha\theta_{\max}<\pi/2$, bound the quadratic form $\rho\cdot\Delta_{\alpha\theta}\rho$ in terms of $\mathcal{L}$ as follows.
    \begin{align*}
        \rho\cdot\Delta_{\alpha\theta}\rho & = \sum_{v_iv_j\in E(\Gamma)}\frac{2\rho_i\rho_j-\cos(\alpha\theta_{ij})\Big(\rho_i^2+\rho_j^2\Big)}{\sin(\alpha\theta_{ij})}\\\\
         & = \sum_{v_iv_j\in E(\Gamma)}\frac{1-\cos(\alpha\theta_{ij})}{\sin(\alpha\theta_{ij})}\big(\rho_i^2+\rho_j^2\big) - \sum_{v_iv_j\in E(\Gamma)}\frac{\big(\rho_i-\rho_j\big)^2}{\sin(\alpha\theta_{ij})}\\
         & \leq \frac{1-\cos(\alpha\theta_{\max})}{\sin(\alpha\theta_{\max})}\sum_ideg(v_i)\rho_i^2 - \frac{1}{\sin(\alpha\theta_{\max})}\sum_{v_iv_j\in E(\Gamma)}\big(\rho_i-\rho_j\big)^2\\
         & = \frac{1-\cos(\alpha\theta_{\max})}{\sin(\alpha\theta_{\max})}\rho\cdot D\rho - \frac{1}{\sin(\alpha\theta_{\max})}\rho\cdot\Delta\rho\\
         & = \frac{\rho}{\sin(\alpha\theta_{\max})}\cdot\Big(\big(1-cos(\alpha\theta_{\max})\big)D-\Delta\Big)\rho\\
         & = \frac{D^{1/2}\rho}{\sin(\alpha\theta_{\max})}\cdot\Big(1-\cos(\alpha\theta_{\max}) - \mathcal{L}\Big)D^{1/2}\rho.
    \end{align*}
    
    Combining Theorem~\ref{eigenvalues} with Proposition~\ref{degree 0} implies $\Delta_{\alpha\theta}$ has at least one positive eigenvalue for each $0<\alpha<\pi/\theta_{max}$. By assumption $1-\cos(\alpha\theta_{\max})<\lambda_1$, so the quadratic form above has exactly one positive eigenvalue. The remaining eigenvalues are all strictly negative, and their corresponding eigenvectors span a hyperplane $W$.
    
    If $\Delta_{\alpha\theta}$ had a second non-negative eigenvalue then there would be a 2-dimensional space on which $\rho\cdot\Delta_{\alpha\theta}\rho$ is non-negative. But such a plane necessarily intersects $W$, on which an upper bound for $\rho\cdot\Delta_{\alpha\theta}\rho$ is negative, a contradiction! Hence $\Delta_{\alpha\theta}$ cannot develop a kernel or a second positive eigenvalue for $0<\alpha<\arccos(1-\lambda_1)/\theta_{\max}$.

    The proof in the case where $\Gamma=K_n$ and $\lambda_1=\frac{n}{n-1}$ is similar. Now we assume $\alpha\theta_{\max} < 2\arctan(n/(n-1))$. We are still in the interval before the first singular degree. We can wtill bound as before
    \[
        \frac{1-\cos(\alpha\theta_{ij})}{\sin(\alpha\theta_{ij})} \leq \frac{1-\cos(\alpha\theta_{\max})}{\sin(\alpha\theta_{\max})}
    \]
    for each edge $v_iv_j$, but now we can only bound $\frac{1}{\sin(\alpha\theta_{ij})}\geq1$. Thus our bound on the quadratic form now reads
    
    \begin{align*}
        \rho\cdot\Delta_{\alpha\theta}\rho & \leq \frac{1-\cos(\alpha\theta_{\max})}{\sin(\alpha\theta_{\max})}\rho\cdot D\rho - \rho\cdot\Delta\rho\\
         & = \tan(\alpha\theta_{\max}/2)\rho\cdot D\rho - \rho\cdot\Delta\rho\\
         & = D^{1/2}\rho\cdot\Big(\tan(\alpha\theta_{\max}/2)-\mathcal{L}\Big)D^{1/2}\rho.
    \end{align*}

    Just as in the previous case, $\rho\cdot\Delta_{\alpha\theta}\rho$ has at least one positive eigenvalue, and our assumed bound on $\alpha$ implies the quadratic form on the right hand side has exactly one positive eigenvalue. Thus $\rho\cdot\Delta_{\alpha\theta}\rho$ cannot have developed a kernel or second positive for $0 < \alpha < \frac{2}{\theta_{\max}}\arctan\left(\frac{n}{n-1}\right)$.
\end{proof}

\subsection{Singular degrees}\label{S euclidean singular}

Let us begin with the extension of Proposition~\ref{constant theta}, where $\theta(e)=\theta_0$ for all $e\in E(\Gamma)$. When $\alpha\theta_0\not\in\pi\mathbb{Z}$, that Proposition equated $1-\cos(\alpha\theta_0)$ with the eigenvalues of $\mathcal{L}$. This leaves out the possibile eigenvalues of 0 and 2.

The normalized Laplacian $\mathcal{L}$ always has an eigenvalue $\lambda=0$, as $D^{1/2}\mathbb{1}$ is in the kernel of $\mathcal{L}$. This corresponds to $\alpha\theta_0=2n\pi$ for $n\in\mathbb{Z}$, which will be discussed in Theorem~\ref{even singular}.

The operator $\mathcal{L}$ may also have an eigenvalue $\lambda=2$, depending on the structure of $\Gamma$. In fact, $2$ is an eigenvalue of $\mathcal{L}$ precisely when $\Gamma$ has no odd cycles. The case $\lambda=2$ corresponds with $\alpha\theta_0=(2n+1)\pi$, which will be discussed in Theorem~\ref{odd singular}.

Before approaching the balancing conditions in Theorems~\ref{even singular} and ~\ref{odd singular} we will first parametrize the space of homogeneous harmonic functions $f:C(\Gamma,\theta)\to\R$ in the case $\alpha\theta_0\in\pi\mathbb{Z}$.

\begin{lemma}\label{constant singular determinism}
    Suppose $\theta(e)=\theta_0$ for all $e\in E(\Gamma)$ and $\alpha\theta_0 = n\pi$ for $n\in\mathbb{Z}$, $n>0$. If $n$ is even, or if $n$ is odd and $\Gamma$ contains no odd cycles, then a homogeneous harmonic function $f:C(\Gamma,\theta)\to\R$ of degree $\alpha$ is uniquely determined by a number $\rho_0$ and a function $c_2:E(\Gamma)\to\R$. If $n$ is odd and $\Gamma$ contains an odd cycle then a homogeneous harmonic function $f:C(\Gamma,\theta)\to\R$ of degree $\alpha$ is uniquely determined by just a function $c_2:E(\Gamma)\to\R$.
\end{lemma}

\begin{proof}
    Begin by arbitrarily orienting all the edges of $\Gamma$, On each face $C(e)\subset C(\Gamma,\theta)$ choose coordinates adapted to the tail of $e$, identifying $C(e)$ with $S_{\theta_0}$ so that the tail corresponds to the positive $x$-axis and the head corresponds to the line $x=y\cot\theta_0$. In these coordinates Lemma~\ref{triggy} says a homogeneous harmonic function is represented by a function
    \[
        u_e(r,\theta) = r^\alpha\big(c_1(e)\cos(\alpha\theta)+c_2(e)\sin(\alpha\theta)\big).
    \]
    
    In the spirit of Lemma~\ref{rho determines f}, we can let $\rho:V(\Gamma)\to\R$ be given by $\rho(v) = f(v,1)$. If edge $e$ has tail $v_1$ and head $v_2$, then Lemma~\ref{BVP} says that $c_1 = \rho(v_1)$ and $c_1 = \rho(v_1)\cos(\alpha\theta_0)$. If $n$ is even this means that $\rho(v_i)=\rho(v_j)$ for each edge $v_iv_j\in E(\Gamma)$. As $\Gamma$ is connected, we must have $\rho\equiv\rho_0$ for some constant $\rho_0$. If $n$ is odd we must have $\rho_i=-\rho_j$ for each edge $v_iv_j\in E(\Gamma)$. If $\Gamma$ contains no odd cycles then it is bipartite, so we may set $\rho=\rho_0$ on one part of $\Gamma$ and $\rho=-\rho_0$ on the other part. But an odd cycle in $\Gamma$ would force $\rho\equiv 0$.
    
    The paragraph above explains the parameter $\rho_0$ in the cases when $n$ is even or $\Gamma$ contains no odd cycle, and its absence in the sace when $n$ is odd and $\Gamma$ contains an odd cycle. Once the constants $c_1(e)$ are so specified, the coefficients $c_2(e)$ in front of $\sin(\alpha\theta)$ can be anything and the function $f$ is well-defined and continuous.
\end{proof}

If $f$ is determined by any constant $\rho_0$ (with $\rho_0=0$ in case $n$ is odd and $\Gamma$ contains an odd cycle) and the function $c_2\equiv0$, then $f$ is already balanced because the derivatives satisfy
\[
    \frac{\partial u_e}{\partial\theta}(r,0) = \frac{\partial u_e}{\partial\theta}(r,\theta_0) = 0.
\]
Considering a more general function $c_2:E(\Gamma)\to\R$, two linear maps become relevant. In case $n$ is even, we define $\partial c_2$, and in case $n$ is odd we define $dc_2$. Separating out a single edge $e_1$, we will define both $\partial$ and $d$ on the function
\[
    c_2(e) = \begin{cases}
        1, & e=e_1\\
        0, & e\neq e_1.
    \end{cases}
\] If $e_1$ has tail $v_1$ and head $v_2$, then define
\[
    \partial c_2(v) = \begin{cases}
        1, & v = v_2\\
        -1, & v = v_1\\
        0, & else
    \end{cases} \qquad\text{and}\qquad dc_2(v) = \begin{cases}
        1, & v = v_1\text{ or }v_2\\
        0, & else.
    \end{cases}
\]
Define the maps $\partial$ and $d$ for arbitrary $c_2:E(\Gamma)\to\R$ by linearity.

\begin{theorem}\label{even singular}
    Suppose $\theta(e)=\theta_0$ for all $e\in E(\Gamma)$ and $\alpha\theta_0=2n\pi$ for some $n\in\mathbb{Z}$, $n>0$. Then the space of balanced homogeneous harmonic functions $f:C(\Gamma,\theta)\to\R$ of degree $\alpha$ has dimension $\#E(\Gamma)-\#V(\Gamma)+2$.
\end{theorem}

\begin{proof}
    According to Lemma~\ref{constant singular determinism} a homogeneous harmonic function is determined by $\rho_0\in\R$ and $c_2:E(\Gamma)\to\R$. The function $f$ is balanced according to Definition~\ref{euclidean balancing} if and only if $\partial c_2 = 0$. It thus remains to compute the dimension of the kernel of $\partial$, or equivalently $\#E(\Gamma)-dim(im\partial)$.
    
    The map $\partial$ turns out to be the boundary operator from homology, whose structure is well understood. For the sake of completeness we provide a description of its image here as the space of functions $g:V(\Gamma)\to\R$ with $\sum_{v\in V(\Gamma)}g(v)=0$. From the structure of $\partial$ it is clear that the image is contained in this space.
    
    Fix a vertex $v_1$, and a collection of paths $p_j$ from $v_1$ to $v_j$ for each other vertex $v_j$. Without loss of generality assume that the orientation along each path points from $v_1$ towards $v_j$. Let $\xi_j(e) = 1$ for each edge $e\in p_j$ and $0$ otherwise, so that
    \[
        \partial \xi_j(v) = \begin{cases}
            1, & v=v_j\\
            -1, & v=v_1\\
            0, & else.
        \end{cases}
    \]
    For any function $g:V(\Gamma)\to\R$ with $\sum_jg(v_j)=0$, let $\xi = \sum_{j\geq2}g(v_j)\xi_j$, so that
    \[
        \partial \xi(v_j) = \begin{cases}
            g(v_j), & j\geq2\\
            -\sum_{i\geq2}g(v_i), & j=1.
        \end{cases}
    \]
    Since $g$ satisfies $\sum_jg(v_j)=0$, we necessarily have $\partial \xi(v_1)=g(v_1)$ also.
    
    The image of $\partial$ thus has dimension $\#V(\Gamma)-1$, so its kernel has dimension $\#E(\Gamma)-\#V(\Gamma)+1$. Combining with the choice of $\rho_0$, the space of balanced homogeneous harmonic functions of degree $\alpha$ is $\#E(\Gamma)-\#V(\Gamma)+2$.
\end{proof}

\begin{theorem}\label{odd singular}
    Suppose $\theta(e)=\theta_0$ for all $e\in E(\Gamma)$ and $\alpha\theta_0=(2n+1)\pi$ for some $n\in\mathbb{Z}$, $n\geq0$. Then the space of balanced homogeneous harmonic functions $f:C(\Gamma,\theta)\to\R$ of degree $\alpha$ has dimension $\#E(\Gamma)-\#V(\Gamma)+2$ if $\Gamma$ contains no odd cycles, and dimension $\#E(\Gamma)-\#V(\Gamma)$ if $\Gamma$ does contain an odd cycle.
\end{theorem}

\begin{proof}
    According to Lemma~\ref{constant singular determinism} a homogeneous harmonic function is determined by $\rho_0\in\R$ and $c_2:E(\Gamma)\to\R$. In case $\Gamma$ contains an odd cycle, though, we must have $\rho_0=0$. The function $f$ is balanced according to Definition~\ref{euclidean balancing} if and only if $dc_2 = 0$. It thus remains to compute the dimension of the kernel of $d$, or equivalently $\#E(\Gamma)-dim(im d)$.
    
    In case $\Gamma$ contains an odd cycle we claim the image of the linear map $d$ consists of all functions $g:V(\Gamma)\to\R$. With the presence of a single odd cycle, one can find a cycle of odd length starting and ending at any particular vertex. For a vertex $v_j$, choose such a path and let $\xi_j(e)=\pm 1$ on the edges of the path in an alternating fashion, so that
    \[
        d\xi_j(v) = \begin{cases}
            2 & v=v_j,\\
            0 & else.
        \end{cases}
    \]
    Defining $\xi=\frac{1}{2}\sum_jg(v_j)\xi_j$, see that $d\xi = g$. In this case we had no choice of $\rho_0$, so the dimension of balanced homogeneous harmonic functions is just
    \[
        dim(ker d) = \#E(\Gamma)-dim(im d) = \#E(\Gamma)-\#V(\Gamma).
    \]
    
    If $\Gamma$ does not contain an odd cycle then the graph is bipartite. Say the parts of $\Gamma$ are $A,B\subset V(\Gamma)$, with $A\cup B=V(\Gamma)$, $A\cap B=\emptyset$, and every edge of $\Gamma$ has one endpoint in $A$ and one endpoint in $B$. Then we claim the image of $d$ consists of those functions $g:V(\Gamma)\to\R$ with
    \[
        \sum_{v\in A}g(v) = \sum_{v\in B}g(v).
    \]
    Assuming $G$ is connected, one can find odd-length paths $p$ from any vertex in $A$ to any vertex in $B$. If the path $p$ joins $v_1\in A$ and $v_2\in B$, then defining $\xi_p(e)=\pm1$ along $p$ in an alternating fashion gives
    \[
        d\xi_p(v) = \begin{cases} 1, & v=v_1\\1, & v=v_2\\0, & else.\end{cases}
    \]
    
    Start with $\xi\equiv 0$ and build up as follows. Enumerate the vertices in $A$ by $v_j$ and the vertices in $B$ by $w_j$, and without loss of generality $\#A=a\geq b=\#B$. First select odd-length paths $p_1,\ldots,p_{a-b}$ from $v_{b+1},\ldots,v_a$ to $w_1$ and add $g(v_j)\xi_{p_j}$ to $\xi$ for $b< j\leq a$. Then choose paths from $w_1$ to $v_1$, $v_1$ to $w_2$, $w_2$ to $v_2$, $\ldots$, $w_b$ to $v_b$. Adding appropriate multiples of the corresponding functions $\xi_p$ to $\xi$ will ensure that $d\xi(v) = g(v)$ for all $v$ except perhaps $v_a$. But the condition that $\sum_{v\in A}g(v) = \sum_{v\in B}g(v)$ will ensure that $d\xi(v_a)=g(v_a)$ as well.
    
    Thus in this case the image of $d$ has dimension $\#V(\Gamma)-1$, so its kernel has dimension $\#E(\Gamma)-\#V(\Gamma)+1$. Combining with the choice of $\rho_0$, the space of balanced homogeneous harmonic functions of degree $\alpha$ has dimension $\#E(\Gamma)-\#V(\Gamma)+2$.
\end{proof}

\begin{remark}\label{rem subdivide}
    Theorem~\ref{even singular} can be seen as a corollary of Theorem~\ref{odd singular} after subdividing edges. In fact, all that is really needed is the statement when $\alpha\theta_0=\pi$. If $\alpha\theta_0=k\pi$ simply subdivide each edge into $k$ pieces to reduce to the case $\alpha\theta_0=\pi$. Lemma~\ref{euclidean smooth} then implies a balanced function on the subdivided complex is smooth at the introduced vertices and is thus the restriction of a function on the un-subdivided complex.
\end{remark}

We now turn to the more complicated situation when $\theta:E(\Gamma)\to(0,\pi)$ is not constant but $\alpha$ is still a singular degree. Let $\Sigma\subset\Gamma$ be the subgraph consisting of all those edges $e\in E(\Gamma)$ with $\alpha\theta(e)\in\pi\mathbb{Z}$, along with their incident vertices. One can see, as in Lemma~\ref{constant singular determinism}, that a homogeneous harmonic function is determined by a function $\rho:V(\Gamma)\to\R$ that lies in the space $B(\alpha)$ described in the previous subsection, along with a function $c_2:E(\Sigma)\to\R$. Then one could describe which $\rho$ and $c_2$ describe a balanced function. But it is simpler to use the strategy described in Remark~\ref{rem subdivide}.

For each edge $e\in E(\Sigma)$, if $\alpha\theta(e)=k\pi$ then subdivide $e$ into $k$ edges, each with $\alpha\theta=\pi$. This turns the original complex $C(\Gamma,\theta)$ into an isometric complex (by abuse of notation also denoted $C(\Gamma,\theta)$) where each edge $e\in E(\Gamma)$ satisfies either $\alpha\theta(e)=\pi$ or $\alpha\theta(e)\not\in\pi\mathbb{Z}$. Now we can describe homogeneous harmonic functions.

\begin{lemma}\label{singular determinism}
    A homogeneous harmonic function $f:C(\Gamma,\theta)\to\R$ of singular degree $\alpha$ is uniquely determined by a function $c_2:E(\Sigma)\to\R$ together with a function $\rho:V(\Gamma)\to\R$ such that $\rho(v_i)=-\rho(v_j)$ for all $v_iv_j\in E(\Sigma)$.
\end{lemma}

\begin{proof}
    Just as in Lemma~\ref{rho determines f}, the function $f$ is uniquely determined on those edges $e\in E(\Gamma)\backslash E(\Sigma)$ by the values $\rho(v)$ on the incident vertices. And just as in Lemma~\ref{constant singular determinism}, the condition $\rho_i=-\rho_j$ together with a choice of $c_{ij}=c_2(v_iv_j)$ on $v_iv_j\in E(\Sigma)$ uniquely determines $f$ in the corresponding face via the representation
    \[
        u(r,\theta) = r^\alpha\Big(\rho_i\cos(\alpha\theta) + c_{ij}\sin(\alpha\theta)\Big).
    \]
\end{proof}

Now the balancing condition is slightly more subtle. On vertices $v_i\in V(\Gamma)\backslash V(\Sigma)$ one must still have
\[
    \sum_{v_j\sim v_i}\frac{\rho_j-\cos(\alpha\theta_{ij})\rho_i}{\sin(\alpha\theta_{ij})} = 0,
\]
just as in Theorem~\ref{euclidean balanced}. Unfortunately these conditions alone seem unlikely to be satisfied for $\rho\not\equiv 0$, given the restrictions on $\rho$ coming from $\Sigma$. Certainly a necessary condition is that the quadratic form $Q(\rho)=\lim_{\alpha'\to\alpha}\rho\cdot\Delta_{\alpha'\theta}\rho$ on the space $B(\alpha)$ from the previous subsection should have a non-trivial kernel. Since the eigenvalues of $\Delta_{\alpha'\theta}$ are strictly increasing with $\alpha'$ the possibility that one passes through 0 precisely at $\alpha'=\alpha$ seems unlikely.

And on those vertices $v_i\in V(\Sigma)$ a combination of the formula from Theorem~\ref{euclidean balanced} and the linear map $d$ from above is necessary, which we'll define as $\sigma:\R^{\#E(\Sigma)}\to\R^{\#V(\Sigma)}$. The balancing condition says
\[
    dc_2(v_i) + \sum_{v_iv_j\in E(\Gamma)\backslash E(\Sigma)}\frac{\rho_j-\cos(\alpha\theta_{ij})\rho_i}{\sin(\alpha\theta_{ij})} = 0.
\]

\begin{proposition}\label{abundance}
    If $\Sigma$ has a connected component that is not a tree but contains no odd cycles, then there are non-trivial balanced homogeneous harmonic functions of degree $\alpha$.
\end{proposition}

\begin{proof}
    To produce a non-trivial balanced homogeneous harmonic function of a singular degree $\alpha$ it is enough, according to Lemma~\ref{singular determinism}, to specify $\rho$ and $c_2:E(\Sigma)\to\R$. We will choose $\rho\equiv 0$, and if $\Sigma_0\subset\Sigma$ is a connected component that is not a tree but contains no odd cycles we set $c_2=0$ on $E(\Sigma)\backslash E(\Sigma_0)$.
    
    We will construct a balanced function by specifying $c_2\not\equiv0$ on $E(\Sigma_0)$ in the kernel of $d$. To prove that such a $c_2$ exists, it suffices to show that the image of $d$ when restricted to functions $c_2:E(\Sigma_0)\to\R$ has dimension at most $\#E(\Sigma_0)-1$. But this is precisely the content of Theorem~\ref{odd singular}; in the absense of odd cycles the image of $d$ acting on $\R^{\#E(\Sigma_0)}$ has dimension $\#V(\Sigma_0)-1$. As $\Sigma_0$ is not a tree, $\#E(\Sigma_0)\geq\#V(\Sigma_0)$, so $d$ has a kernel and a balanced homogeneous harmonic function exists.
\end{proof}

\subsection{Eigenvalue problems on graphs}\label{S eigenvalues}

The balanced homogeneous harmonic functions of this section can be reinterpreted in terms of an eigenvalue problem on the graph $\Gamma$. Returning to the notation of Definition~\ref{homogeneous}, write $f:C(\Gamma,\theta)\to\R$ as
\[ 
    f(x,t) = \rho(x)t^\alpha.
\]
Here $x$ is a point on $\Gamma$ and $t\in\hl$. On each edge $e$ determine coordinates by a map $\psi_e:e\to[0,\theta(e)]\subset\R$, and define $\rho_e:[0,\theta(e)]\to\R$ by
\[
    \rho_e(s) = \rho\circ\psi_e^{-1}(s).
\]

Now the original function $f:C(\Gamma,\theta)\to\R$ is harmonic if and only if each $\rho_e$ is an eigenfunction of the second derivative operator, namely $\rho''_e=-\alpha^2\rho_e$. And the original function $f$ is balanced if and only if, for each $v\in V(\Gamma)$,
\[
    \sum_{v\in e}(-1)^{\psi_e(v)/\theta(e)}\rho_e'(\psi_e(v)) = 0.
\]
Thus a balanced homogeneous harmonic function $f$ corresponds to a solution of the eigenvalue problem $\rho''_e=-\alpha^2\rho_e$ for all $e\in E(\Gamma)$ with the boundary conditions stated above at each $v\in V(\gamma)$.

The functions $\rho:\Gamma\to\R$ can also be studied variationally. Define the energy $E[\rho]$ and $L^2$ norm $\norm{\rho}^2$ as follows.
\[
    E[\rho] = \sum_{e\in E(\Gamma)} \int_0^{\theta(e)}\abs{\rho'_e(s)}^2ds \qquad\text{and}\qquad \norm{\rho}^2 = \sum_{e\in E(\Gamma)}\int_0^{\theta(e)} \abs{\rho_e(s)}^2ds.
\]
The main result of this subsection relates the eigenvalue problem described above to a Rayleigh quotient involving these two functions.

\begin{theorem}\label{rayleigh}
    The solutions $\rho$ to the eigenvalue problem $\rho''_e=-\alpha^2\rho_e$ for all $e\in E(\Gamma)$ with boundary conditions coming from the balancing condition on each $v\in V(\gamma)$ are precisely the critical points of the Rayleigh quotient
    \[
        R[\rho] = \frac{E[\rho]}{\norm{\rho}^2}.
    \]
\end{theorem}

\begin{proof}
    First compute the first variation formula for the Rayleigh quotient. If $\rho_t$ is a variation of maps with $\rho=\rho_0$ and $\dot{\rho} = \frac{\partial}{\partial t}\vert_{t=0}\rho_t$, compute
    \begin{align*}
        \frac{\partial}{\partial t}\vert_{t=0}R[\rho_t] & = \frac{\partial}{\partial t}\vert_{t=0}\frac{E[\rho_t]}{\norm{\rho_t}^2}\\
         & = \frac{2}{\norm{\rho}^4}\Bigg(\norm{\rho}^2\sum_e\int_0^{\theta(e)}\rho'_e(s)\dot{\rho}'_e(s)ds - E[\rho]\sum_e\int_0^{\theta(e)}\rho_e(s)\dot{\rho}_e(s)ds\Bigg)\\
         & = \frac{2}{\norm{\rho}^4}\sum_e\Bigg(\norm{\rho}^2\Big[\rho'_e(s)\dot{\rho}_e(s)\Big]_0^{\theta(e)} - \int_0^{\theta(e)}\dot{\rho}_e(s)\Big(\norm{\rho}^2\rho''_e(s) + E[\rho]\rho_e(s)\Big)ds\Bigg).
    \end{align*}
    
    Rewriting the boundary term we find
    \[
        \sum_e \big(\rho'_e(\theta(e))\dot{\rho}_e(\theta(e)) - \rho'_e(0)\dot{\rho}_e(0)\big) = \sum_v\dot{\rho}(v)\sum_{v\in e}(-1)^{\psi_e(v)/\theta(e)}\rho'_e(\psi_e(v)).
    \]
    Thus $\rho$ is a critical point for the Rayleigh quotient if and only if
    \[
        \rho''_e = -R[\rho]\rho_e\text{ for all }e\in E(\Gamma)\qquad\text{and}\qquad \sum_{v\in e}(-1)^{\psi_e(v)/\theta(e)}\rho'_e(\psi_e(v)) = 0\text{ for all }v\in V(\Gamma).
    \]
    That is, $\rho$ is an eigenfunction of eigenvalue $-\alpha^2=-R[\rho]$ and satisfies the balancing boundary conditions.
\end{proof}

\section{Balanced harmonic maps into $k$-pods}\label{S k-pod target}

After subdividing the domain if necessary a homogeneous map $f:C(\Gamma,\theta)\to C(k)$ maps each face $C(e)\subset C(\Gamma,\theta)$ to a single edge $C(v_k)\subset C(k)$. The map $f$ is harmonic if its restriction to each face, when viewed in coordinates, is a harmonic function $u:S_{\theta(e)}\to\hl$. According to Proposition~\ref{non-neg} such a map exists if and only if $0<\alpha\theta(e)\leq\pi$ for each edge $e\in E(\Gamma)$.

The balancing condition of Definition~\ref{1-d balancing} imposes further restrictions on which degrees $\alpha$ can occur for harmonic maps from smooth cones into $k$-pods. In fact the degree $\alpha$ and the geometry of the cone are closely related.

\begin{theorem}\label{1-d harmonic}
    A balanced homogeneous harmonic map $f:C(C_{n'},\theta')\to C(k)$ of degree $\alpha$ exists if an only if the cone is a subdivision of $C(C_n,\theta)$ where $\theta(e)=\theta_0$ for each edge $e\in E(C_n)$, and
    \[
        \alpha\theta_0=\pi.
    \]
\end{theorem}

\begin{proof}
    First suppose that there is a balanced homogeneous harmonic map $f:C(C_{n'},\theta')\to C(k)$ and $f(v_0,1)=0$ is the vertex of $C(k)$ for some fixed $v_0\in V(C_{n'})$. Enumerate the vertices of $C_{n'}$ in order as $\{v_0,v_2,\ldots,v_{n'-1}\}$, so $v_i\sim v_{(i+1)\pmod{n}}$ for each $i$. In each face $C(v_iv_{i+1})\subset C(C_{n'},\theta')$ take coordinates adapted to $v_i$ and represent $f$ in those coordinates by a harmonic function $u_i:S_{\theta(v_iv_{i+1})}\to\hl$.
    
    Since $f(v_0,1)=0$ is the vertex of the $k$-pod $C(k)$ we must have $\rho(v_0)=0$. The map $u_0$ that represents $f$ in $C(v_0v_1)$ thus has the form
    \[
        u_0(r,\theta) = c_1r^\alpha\sin(\alpha\theta).
    \]
    If $\alpha\theta(v_0v_1) = \pi$ then $u_0(r,\theta(v_0v_1)) = 0$ as well. Otherwise $u_0(r,\theta(v_0v_1))\neq0$ and according to Lemma~\ref{k-pod smooth} the map $u_1$ must have the form
    \[
        u_1(r,\theta) = c_1r^\alpha\sin\big(\alpha\theta+\alpha\theta(v_0v_1)\big).
    \]
    Continuing inductively in this fashion, as long as $u_{i-1}(r,\theta(v_{i-1}v_i))\neq 0$ Lemma~\ref{k-pod smooth} says the map $u_i$ representing $f$ in $C(v_iv_{i+1})$ must have the form
    \[
        u_i(r,\theta) = c_1r^\alpha\sin\big(\alpha\theta+\alpha\theta(v_0v_1) + \cdots + \alpha\theta(v_{i-1}v_i)\big).
    \]
    
    Since $f$ can only map an edge of $C(C_{n'},\theta')$ to the vertex of $C(k)$, some $u_m$ must vanish at $(r,\theta(v_mv_{m+1}))$. In this case we have
    \[
        \alpha\sum_{i=1}^r \theta(v_iv_{i+1}) = \pi.
    \]
    We can ``un-subdivide" the first $r$ edges of $c_{n'}$ to form an edge of a new graph with $\alpha\theta=\pi$ on that edge. Repeating the same argument starting at $v_{r+1}$, and at each successive vertex of $C_{n'}$ that $f$ maps to the vertex of $C(k)$, one sees that $C(C_{n'},\theta')$ is a subdivision of some $C(C_n,\theta)$ with $\alpha\theta(e)=\pi$ for each $e\in E(C_n)$.
    
    It remains to see that a balanced homogeneous harmonic map of degree $\alpha$ does exist from $C(C_n,\theta)$ when $\alpha\theta(e)=\pi$ for each $e\in E(C_n)$. Let $\theta_0 = \pi/\alpha$ be the common angle measure in each face of $C(C_n,\theta)$. In each face of $C(C_n,\theta)$ let $f$ be represented in coordinates by the function
    \[
        u(r,\theta) = cr^\alpha\sin(\alpha\theta).
    \]
    Then $f$ will map \emph{each edge} of $C(C_n,\theta)$ into the vertex of $C(k)$. And the balancing condition is satisfied precisely because the constant $c$ in each face is the same.
\end{proof}

\begin{remark}
    The process of ``un-subdividing" described in this proof may result in a cone $C(C_n,\theta)$ with $\theta(e)\geq\pi$. This is not consistent with Definition~\ref{metric} but it causes only notational difficulty. In particular, if $\theta(e)=\theta_0\geq\pi$ one can identify the face $C(e)$ with the upper half-plane endowed with a different metric. Indeed, let
    \[
        S = \{(x,y)\in\R^2 \vert y\geq 0\} \quad\text{with metric}\quad ds^2 = dr^2 + \frac{\theta_0^2r^2}{\pi^2}d\theta^2.
    \]
    This $S$ has an angle of $\theta_0$ at the origin, and the harmonic functions of degree $\alpha$ defined on $S$ have the form
    \[
        u(r,\theta) = r^\alpha\big(c_1\cos(\theta)+c_2\sin(\theta)\big).
    \]
    In order that $u$ stays positive one must have $c_1=0$ and $c_2>0$, which is consistent with the functions described in the above proof.
\end{remark}

\begin{remark}\label{max princ}
    According to our assumption from Section~\ref{S 1-d balancing}, if $v\in V(C_n)$ has incident edges $e_1$ and $e_2$, and $f(v,1)$ is the vertex of $C(k)$, then the images of $f\vert_{C(e_1)}$ and $f\vert_{C(e_2)}$ should be distinct edges of $C(k)$. This is a common assumption in applications, and is consistent with the maximum principle for harmonic maps that would prohibit a local minimum value of $0$ along an edge. With this additional restriction, the number $n$ from the above proof must be at least 2, or at least $3$ if $f$ takes advantage of the singular nature of $C(k)$.
\end{remark}

The application we list here was established in \cite{gromov-schoen}.

\begin{corollary}
    Any homogeneous harmonic map $f:\R^2\to C(k)$ into a $k$-pod either has degree $1$ or degree $\alpha\geq 3/2$. Moreover if $C(C_n,\theta)$ is a smooth cone whose vertex is positively curved, then any homogeneous harmonic map $f:C(C_n,\theta)\to C(k)$ has degree $\alpha>1$, and if the total angle of $C(C_n,\theta)$ is at most $3\pi$ then any map that visits at least three edges of $C(k)$ has degree $\alpha\geq1$.
\end{corollary}

\begin{proof}
    We begin by remarking that $\R^2$ is isometric to $C(C_{n'},\theta')$ whenever
    \[
        \sum_{e\in E(C_{n'})}\theta'(e) = 2\pi.
    \]
    According to Theorem~\ref{1-d harmonic}, the degree $\alpha$ of a homogeneous harmonic map $f:C(C_n,\theta)\to C(k)$ must satisfy
    \[
        \alpha\sum_{e\in E(C_{n'})}\theta'(e) = n\pi,
    \]
    where $n$ is the number of edges of $C(k)$ visited by $f$ counted with multiplicity. Following Remark~\ref{max princ} we must have $n\geq 2$, or $n\geq3$ if $f$ visits at least 3 edges of $C(k)$.
    
    If the vertex of $C(C_{n'},\theta')$ is flat, this means $\alpha = n/2$. When $n=2$ we have $\alpha=1$, and when $n\geq3$ we have $\alpha\geq3/2$.
    
    If the vertex of $C(C_{n'},\theta')$ is positively curved then $\sum_e\theta'(e)<2\pi$ so
    \[
        \alpha = \frac{n\pi}{\sum_{e\in E(C_{n'})}\theta'(e)} > n/2.
    \]
    Again by Remark~\ref{max princ} $n\geq2$ so that $\alpha>1$.
    
    And in fact if
    \[
        \sum_{e\in E(C_{n'})}\theta'(e) \leq3\pi\qquad\text{and}\qquad n\geq 3,
    \]
     then
     \[
        \alpha = \frac{n\pi}{\sum\theta'(e)} \geq 1.
     \]
\end{proof}

\subsection{$p$-harmonic maps}

Just as in the previous subsection, we consider $p$-harmonic maps $f:C(C_n,\theta)\to C(k)$. The basic object of study is analogous to the non-negative harmonic functions described in Proposition~\ref{non-neg}. The existence of homogeneous $p$-harmonic functions with Dirichlet boundary conditions in a plane sector were studied in \cite{tolksdorf} and later in \cite{veron-porretta}. Akman, Lewis, and Vogel describe in \cite{akman-lewis-vogel} the relationship between the sector $S_{\theta_0}$ and the degree $\alpha$ of homogeneity of $p$-harmonic functions described in \cite{tolksdorf} and \cite{veron-porretta}.

Translating the description of \cite{akman-lewis-vogel} into the present notation, for each $p>1$ and each angle $0<\theta_0<2\pi$ there is a unique degree $\alpha>0$ and a unique (up to scaling) homogeneous $p$-harmonic function $\phi_\alpha:S_{\theta_0}\to\hl$ of degree $\alpha$ with $\phi_\alpha(r,0) = \phi_\alpha(r,\theta_0) = 0$. (There is also a unique degree $\alpha<0)$ with an associated $p$-harmonic function, but we will not consider that here). The angle $\theta_0$ and the degree $\alpha$ are related by the equation
\[
    1-\frac{\theta_0}{\pi} = \frac{\alpha-1}{\sqrt{\alpha^2 + \frac{2-p}{p-1}\alpha}}.
\]

\begin{lemma}\label{p decreasing}
    The relationship between $\theta_0$ and $\alpha$ is monotone. That is, if $\theta_0<\theta_1$ support homogeneous $p$-harmonic functions of degrees $\alpha_0$ and $\alpha_1$, respectively, then $\alpha_0>\alpha_1$.
\end{lemma}

\begin{proof}
    Simply differentiate the equation relating $\theta_0$ and $\alpha$. This yields
    \[
        -\frac{1}{\pi} = \frac{p\alpha+2-p}{2(p-1)\left(\alpha^2+\frac{2-p}{p-1}\alpha\right)^{3/2}}\cdot\frac{d\alpha}{d\theta_0}.
    \]
    Now, in order that the square root in the equation relating $\theta$ and $\alpha$ is real, we must have $\alpha>\frac{p-2}{p-1}$. This rearranges to say that $p\alpha+2-p>\alpha$, so the derivative computation above rearranges to
    \[
        \frac{d\alpha}{d\theta_0} < -\frac{2(p-1)\left(\alpha^2+\frac{2-p}{p-1}\alpha\right)^{3/2}}{\alpha\pi} < 0.
    \]
\end{proof}

\begin{theorem}
    Any homogeneous $p$-harmonic map $f:\R^2\to C(k)$ into a $k$-pod either has degree $1$ or degree $\alpha\geq 9/8$. Moreover if $C(C_n,\theta)$ is a smooth cone whose vertex is positively curved, then any homogeneous $p$-harmonic map $f:C(C_n,\theta)\to C(k)$ has degree $\alpha>1$, and if the total angle of $C(C_n,\theta)$ is at most $3\pi$ then any map that visits at least three edges of $C(k)$ has degree $\alpha\geq1$.
\end{theorem}

\begin{proof}
    Just as in Theorem~\ref{1-d harmonic} we can construct homogeneous $p$-harmonic maps $f:C(C_n,\theta)\to C(k)$ essentially (i.e. up to subdivision) only when $\theta(e)=\theta_0$ for all $e\in E(C_n)$. In each face $C(e)$, $f$ is represented in coordinates by $c(e)\phi_\alpha$, where $\alpha$ is the degree related to the angle $\theta_0$ in the above equation. Then the balancing condition of Definition~\ref{1-d balancing} says that $f$ is balanced if and only if the constants $c(e)$ do not depend on $e\in E(\Gamma)$, i.e. $c(e)=c$ for all $e$.

    If the domain is isometric to $\R^2$, i.e. if the vertex of $C(C_n,\theta)$ is flat, that means
    \[
        \sum_{e\in E(C_n)}\theta(e) = n\theta_0 = 2\pi
    \]
    
    When $n=2$ and $\theta_0=\pi$ the corresponding degree of a $p$-harmonic function $u:S_\pi\to\hl$ is $\alpha=1$. If $n\geq3$ then $\theta_0\leq 2\pi/3$. According to Lemma~\ref{p decreasing} smaller angles $\theta_0$ correspond to larger degrees $\alpha$, so the smallest degree happens when $\theta_0=2\pi/3$. Solving the relationship between $\theta_0$ and $\alpha$ in the case $\theta_0=2\pi/3$ gives
    \[
        \alpha = \frac{17p-16 + \sqrt{p^2+32p-32}}{16(p-1)}.
    \]
    This is a decreasing function of $p$ so it is bounded below by its limit as $p\to\infty$, namely $\alpha\geq 9/8$.
    
    And if the vertex of $C(C_n,\theta)$ is positively curved then $\sum_e\theta(e) = n\theta_0<2\pi$. In this case we have $\theta_0<\pi$, so by Lemma~\ref{p decreasing} we have $\alpha>1$.
    
    In fact if $\sum_e\theta(e)=n\theta_0\leq3\pi$ and the map $f$ visits at least 3 edges of $C(k)$ then we must have $n\geq3$ so that $\theta_0\leq\pi$, and again this means $\alpha\geq1$.
\end{proof}

\section{Balanced harmonic maps between cones}\label{S singular target}

Consider a homogeneous map $f:C(\Gamma,\theta)\to C(\Gamma,\varphi)$, and assume $f$ maps each edge and face of $C(\Gamma)$ to itself. The map $f$ is harmonic if its restriction to each face of $C(\Gamma,\theta)$ is a harmonic map between sectors of the plane. In this section we investigate the balanced (as in Definition~\ref{2-d balancing}) homogeneous harmonic maps $f:C(\Gamma,\theta)\to C(\Gamma,\varphi)$. According to Proposition~\ref{maps between sectors} such a map can only exist if $0<\alpha\theta(e)\leq\pi$ for each $e\in E(\Gamma)$. The analysis will be different if $\alpha\theta(e)=\pi$ for some edge $e$ (necessarily $\theta(e)=\theta_{\max}$); we will call this the singular degree.

We will focus first on the case of non-singular $\alpha$. We begin by constructing homogeneous harmonic maps $f:C(\Gamma,\theta)\to C(\Gamma,\varphi)$, and later investigating which of those maps are balanced.

\begin{lemma}
    A homogeneous harmonic map $f:C(\Gamma,\theta)\to C(\Gamma,\varphi)$ of degree $0<\alpha<\pi/\theta_{\max}$ is uniquely determined by a function $\rho:V(\Gamma)\to\hl$.
\end{lemma}

\begin{proof}
    For a fixed edge $e\in E(\Gamma)$ consider the faces $C(e)\subset C(\Gamma,\theta)$ and $C(e)\subset C(\Gamma,\varphi)$, which are isometric to $S_{\theta(e)}$ and $S_{\varphi(e)}$ respectively. According to Proposition~\ref{maps between sectors} a homogeneous harmonic map $u:S_{\theta(e)}\to S_{\varphi(e)}$ is uniquely determined by the values $\abs{u(1,0)}$ and $\abs{u(1,\theta_0)}$ (using polar coordinates $(r,\theta)$ in $S_{\theta(e)}$).
    
    Fix a function $\rho:V(\Gamma)\to\hl$. For each edge $e=(v_0v_1)\in E(\Gamma)$, fix coordinates in the face $C(e)$ adapted to $v_0$ for each metric, so that $C(e)$ is identified with $S_{\theta(e)}$ or $S_{\varphi(e)}$ and the edge $C(v)$ is identified with the positive $x$-axis. Define $f$ in $C(e)$ to be represented in these coordinates by the unique homogeneous harmonic map $u:S_{\theta(e)}\to S_{\varphi(e)}$ with $u(1,0)=\big(\rho(v_0),0\big)$ and $u(1,\theta(e))=\big(\rho(v_1)\cos(\varphi(e)),\rho(v_1)\sin(\varphi(e))\big)$.
    
    Along each edge $C(v)$ of $C(\Gamma)$ the map $f$ is defined independently of the choice of incident edge. Hence $f$ is continuous over the edges of $C(\Gamma,\theta)$ and harmonic in each face.
\end{proof}

For ease of notation, we will enumerate the vertices of $\Gamma$ by $\{v_i\}$. Then a function $\rho:V(\Gamma)\to\hl$ can be encoded by the vector $(\rho_i=\rho(v_i))$. If $v_iv_j\in E(\Gamma)$ is an edge, we will also denote by $\theta_{ij}=\theta(v_iv_j)$ and $\varphi_{ij} = \varphi(v_iv_j)$.

\begin{theorem}
    The homogeneous harmonic map $f:C(\Gamma,\theta)\to C(\Gamma,\varphi)$ of degree $0<\alpha<\pi/\theta_{\max}$ determined by $\rho:V(\Gamma)\to\hl$ is balanced if and only if for each vertex $v_i$,
    \[
        \sum_{v_j\sim v_i} \frac{\cos(\varphi_{ij})\rho_j-\cos(\alpha\theta_{ij})\rho_i}{\sin(\alpha\theta_{ij})}=0.
    \]
\end{theorem}
Here $v\sim w$ means that vertices $v$ and $w$ are adjacent, so there is an edge $vw\in E(\Gamma)$.

\begin{proof}
    In each face $C(v_iv_j)$ choose coordinates adapted to $v_i$ for each metric, so that the face is identified with $S_{\theta_{ij}}$ or $S_{\varphi_{ij}}$ and $C(v_i)$ is identified with the positive $x$-axis. According to Proposition~\ref{maps between sectors}, in these coordinates $f\vert_{C(v_iv_j)}$ is represented by
    \[
        u_j(r,\theta) = r^\alpha\left(\rho_i\cos(\alpha\theta) + \frac{\rho_j\cos(\varphi_{ij})-\rho_i\cos(\alpha\theta_{ij})}{\sin(\alpha\theta_{ij})}\sin(\alpha\theta),\frac{\sin(\varphi_{ij})}{\sin(\alpha\theta_{ij})}\rho_j\sin(\alpha\theta)\right).
    \]
    And the balancing formula of Definition~\ref{2-d balancing} says that $f$ is balanced along $C(v_i)$ if
    \[
        \sum_{v_j\sim v_i} \frac{\partial \chi_j}{\partial\theta}(r,0) = 0,
    \]
    where $r^\alpha\chi(\theta)$ is the $x$-coordinate of $u(r,\theta)$.
    
    Combining these two formulas, we see that $f$ is balanced along $C(v_i)$ if
    \[
        \alpha\sum_{v_j\sim v_i} \frac{\cos(\varphi_{ij})\rho_j-\cos(\alpha\theta_{ij})\rho_i}{\sin(\alpha\theta_{ij}))} = 0.
    \]
    In order to be balanced, $f$ must be balanced along each edge of $C(\Gamma,\theta)$, so every sum of this form must vanish.
\end{proof}

Another way to interpret this result is to say a vector $(\rho_j)$ determines a balanced homogeneous harmonic function $f:C(\Gamma,\theta)\to C(\Gamma,\varphi)$ of degree $0<\alpha<\pi/\theta_{\max}$ if it is in the kernel of the matrix $\Delta^\varphi_{\alpha\theta}$ with entries
\[
    (\Delta^\varphi_{\alpha\theta})_{ij} = \begin{cases}
        -\sum_{v_k\sim v_i}\cot(\alpha\theta_{ik}), & i=j\\
        \cos(\varphi_{ij})\csc(\alpha\theta_{ij}), & v_i\sim v_j\\
        0 & v_i\not\sim v_j.
    \end{cases}
\]
For only certain values of $\alpha$ will $\Delta^\varphi_{\alpha\theta}$ have a non-trivial kernel. That is, only certain degrees $\alpha$ admit balanced homogeneous harmonic maps.

The situation for harmonic maps between cones is much more rigid than that of harmonic functions. In special cases the degrees $\alpha$ and the corresponding balanced homogeneous harmonic maps are easy to describe.

\begin{proposition}\label{constant theta and phi}
    Suppose $\theta(e)=\theta_0$ and $\varphi(e)=\varphi_0$ for all $e\in E(\Gamma)$. Then there exist balanced harmonic maps $f:C(\Gamma,\theta)\to C(\Gamma,\varphi)$ of degree $0<\alpha<\pi/\theta_0$ if and only if $\alpha\theta_0=\varphi_0$. If $\alpha\theta_0=\varphi_0=\pi/2$ then every function $\rho:V(\Gamma)\to\hl$ determines such an $f$, but if $\alpha\theta_0=\varphi_0\neq\pi/2$ then only the constant vectors $\rho=c\mathbb{1}$ with $\rho_i=c$ for all $i$ do.
\end{proposition}

\begin{proof}
    First suppose $\varphi_0=\pi/2$. Then $\cos(\varphi_0) = 0$ so the matrix $\Delta^\varphi_{\alpha\theta}$ is diagonal with diagonal entries $deg(v_i)\cot(\alpha\theta_0)$. A diagonal operator only has a kernel if some of its diagonal entries are 0. In this case the only possibility is if $\cot(\alpha\theta_0)=0$, i.e. $\alpha\theta_0=\pi/2$. In this case $\Delta^\varphi_{\alpha\theta_0} = 0$, so all vectors $(\rho_i)$ are in its kernel.
    
    Now suppose $\varphi_0\neq\pi/2$. The matrix $\Delta^\varphi_{\alpha\theta}$ can be rewritten
    \[
        \Delta^\varphi_{\alpha\theta} = \csc(\alpha\theta_0)\Big(\cos(\varphi_0)A - \cos(\alpha\theta_0)D\Big).
    \]
    A vector $\rho$ is in the kernel of this matrix if and only if
    \begin{align*}
        \cos(\varphi_0)A\rho & = \cos(\alpha\theta_0)D\rho\\
        D^{-1}A\rho & = \frac{\cos(\alpha\theta_0)}{\cos(\varphi_0)}\rho\\
        D^{-1/2}\Big(1-D^{-1/2}AD^{-1/2}\Big)D^{1/2}\rho & = \frac{\cos(\varphi_0)-\cos(\alpha\theta_0)}{\cos(\varphi_0)}\rho\\
        \mathcal{L}D^{1/2}\rho & = \frac{\cos(\varphi_0)-\cos(\alpha\theta_0)}{\cos(\varphi_0)}D^{1/2}\rho
    \end{align*}
    Thus $D^{1/2}\rho$ is an eigenvector of $\mathcal{L}$ with eigenvalue $\lambda = 1-\frac{\cos(\alpha\theta_0)}{\cos(\varphi_0)}$.
    
    If $\alpha\theta_0 \neq\varphi_0$ then $\lambda\neq 0$. But the eigenvectors of $\mathcal{L}$ with non-zero eigenvalue are perpendicular to the null-vector $w=D^{1/2}\mathbb{1}$ with $w_i = \sqrt{deg(v_i)}$ for each $i$. In other words,
    \[
        0 = D^{1/2}\rho\cdot w = \sum_i deg(v_i)\rho_i.
    \]
    But in order to determine a balanced homogeneous harmonic map, all coordinates of $\rho$ must be non-negative! For $\rho\not\equiv0$ the above equation precludes this possiblity, so the only case left to consider is $\alpha\theta_0=\varphi_0$.
    
    If $\alpha\theta_0=\varphi_0$ then a vector $\rho$ in the kernel of $\Delta^\varphi_{\alpha\theta}$ must satisfy $D^{1/2}\rho = cw =cD^{1/2}\mathbb{1}$, so $\rho=c\mathbb{1}$.
\end{proof}

For metrics given by more general functions $\theta,\varphi:E(\Gamma)\to(0,\pi)$, the question of which degrees $\alpha$ admit balanced homogeneous harmonic functions is more subtle. But we can achieve a lower bound on such $\alpha$ more easily than we did for functions into $\R$.

\begin{theorem}\label{neg def}
    If there is a balanced homogeneous harmonic map $f:C(\Gamma,\theta)\to C(\Gamma,\varphi)$ then
    \[
        \alpha\geq \min\left(\frac{\pi}{2\theta_{\max}},\min\left\{\frac{\varphi(e)}{\theta(e)}\vert e\in E(\Gamma)\right\}\right).
    \]
\end{theorem}

\begin{proof}
    We proceed by the contrapositive. If $\alpha<\pi/2\theta_{\max}$ and $\alpha\theta(e)<\varphi(e)$ for each $e\in E(\Gamma)$, we will show that no balanced homogeneous harmonic map of degree $\alpha$ can exist. Under these assumptions on $\alpha$ we have for each edge $v_iv_j$,
    \[
        \cos(\alpha\theta_{ij}) > \cos(\varphi_{ij}) \quad\text{and}\quad \cos(\alpha\theta_{ij})>0.
    \]
    Now bound the quadratic form associated to $\Delta^\varphi_{\alpha\theta}$ as follows.
    \begin{align*}
        \rho\cdot\Delta^\varphi_{\alpha\theta}\rho & = \sum_{v_iv_j\in E(\gamma)}\frac{2\rho_i\rho_j\cos(\varphi_{ij}) - (\rho_i^2+\rho_j^2)\cos(\alpha\theta_{ij})}{\sin(\alpha\theta_{ij})}\\
         & \leq -\sum_{v_iv_j\in E(\Gamma)}\cot(\alpha\theta_{ij})\big(\rho_i-\rho_j\big)^2 \leq 0.
    \end{align*}
    This quadratic form is negative definite; if $\rho$ is not the zero vector then $\rho\cdot\Delta^\varphi_{\alpha\theta}\rho<0$. Thus no non-trivial balanced homogeneous harmonic maps of degree $\alpha$ can exist.
\end{proof}

We can further study the quadratic form determined by $\Delta^\varphi_{\alpha\theta}$ just as we did in the case of functions to $\R$.

\begin{lemma}\label{2-d increasing}
    For each fixed $\rho=(\rho_i)\not\equiv0$ the quadratic form $\rho\cdot\Delta^\varphi_{\alpha\theta}\rho$ is strictly increasing in $\alpha$.
\end{lemma}

\begin{proof}
    We computed the quadratic form in the previous result.
    \[
        \rho\cdot\Delta^\varphi_{\alpha\theta}\rho = \sum_{v_iv_j\in E(\gamma)}\Big(2\rho_i\rho_j\cos(\varphi_{ij})\csc(\alpha\theta_{ij}) - (\rho_i^2+\rho_j^2)\cot(\alpha\theta_{ij})\Big).
    \]
    Its derivative can now be bounded as follows.
    \begin{align*}
        \frac{\partial}{\partial\alpha}\Big(\rho\cdot\Delta^\varphi_{\alpha\theta}\rho\Big) & = \sum_{v_iv_j\in E(\Gamma)}\theta_{ij}\csc^2(\alpha\theta_{ij})\Big(\rho_i^2+\rho_j^2 - 2\rho_i\rho_j\cos(\varphi_{ij})\cos(\alpha\theta_{ij})\Big)\\
         & \geq \sum_{v_iv_j\in E(\Gamma)}\theta_{ij}\csc^2(\alpha\theta_{ij})\big(\rho_i-\rho_j\big)^2 \geq 0.
    \end{align*}
    
    Moreover suppose $\frac{\partial}{\partial\alpha}\Big(\rho\cdot\Delta^\varphi_{\alpha\theta}\rho\Big)=0$ at some fixed $\rho=(\rho_i)$ and some fixed $\alpha$. Then $\rho_i\rho_j\cos(\varphi_{ij})\cos(\alpha\theta_{ij}) = \rho_i\rho_j$ for each $v_iv_j\in E(\Gamma)$, and also
    \[
        \sum_{v_iv_j\in E(\Gamma)}\theta_{ij}\csc^2(\alpha\theta_{ij})\Big(\rho_i-\rho_j\Big)^2 = 0.
    \]
    The expression on the left is clearly non-negative, so for equality to hold $\rho$ must be constant. And it follows from $0<\varphi_{ij},\alpha\theta_{ij}<\pi$, $\cos(\varphi_{ij})\cos(\alpha\theta_{ij})\neq1$ that $\rho_i\rho_j=0$ for all $v_iv_j\in E(\Gamma)$. Since $\rho$ is constant, we have $\rho\equiv0$. That is, the only way that $\frac{\partial}{\partial\alpha}\Big(\rho\cdot\Delta_{\alpha\theta}\rho\Big)=0$ at any $\alpha$ is if $\rho\equiv 0$.
\end{proof}

\begin{theorem}\label{2-d eigenvalues}
    In the interval $0<\alpha<\pi/\theta_{\max}$ the eigenvalues of $\Delta^\varphi_{\alpha\theta}$ are strictly increasing functions of $\alpha$.
\end{theorem}

\begin{proof}
    The entries of the matrix $\Delta^\varphi_{\alpha\theta}$ are analytic in $\alpha$ in the interval $0<\alpha<\pi/\theta_{\max}$. Moreover, $\Delta^\varphi_{\alpha\theta}$ is symmetric for each real $\alpha$. A result in the perturbation theory of eigenvalue problems (see \cite{rellich-berkowitz} Section 1.1 or \cite{kato} Theorem 6.1) says that the eigenvalues and eigenvectors of $\Delta_{\alpha\theta}$ are also analytic in $\alpha$. Namely, there are $\lambda_j(\alpha)\in\R$ and $w_j(\alpha)\in\R^{\#V(\Gamma)}$ depending analytically on $\alpha$ so that
    \[
        \Delta^\varphi_{\alpha\theta}w_j(\alpha) = \lambda_j(\alpha)w_j(\alpha).
    \]
    
    Without loss of generality, $w_j(\alpha)$ is a unit vector for each $j$ and each $\alpha$. The eigenvalues can be recovered from the eigenvectors via the formula
    \[
        \lambda_j(\alpha) = w_j(\alpha)\cdot\Delta^\varphi_{\alpha\theta}w_j(\alpha).
    \]
    Differentiating this identity with respect to $\alpha$ yields
    \[
        \lambda'_j(\alpha) = 2w'_j(\alpha)\cdot\Delta^\varphi_{\alpha\theta}w_j(\alpha) + w_j(\alpha)\cdot\frac{\partial\Delta^\varphi_{\alpha\theta}}{\partial\alpha}w_j(\alpha).
    \]
    
    As each $w_j(\alpha)$ is a unit vector, it lies in the unit sphere of $\R^{\#V(\Gamma)}$. Hence $w'(\alpha)$ is perpendicular to $w_j(\alpha)$. But $\Delta_{\alpha\theta}w_j(\alpha) = \lambda_j(\alpha)w_j(\alpha)$ is parallel to $w_j(\alpha)$, so the first term in $\lambda'(\alpha)$ vanishes. And the second term is strictly positive by Lemma~\ref{2-d increasing}. Thus
    \[
        \lambda'_j(\alpha)>0.
    \]
\end{proof}

An immediate consequence of this Theorem is that one can count the number o degrees $\alpha$ admitting balanced homogeneous harmonic functions just as in the case of functions to $\R$. If $\alpha\theta_{ij}=\pi-\epsilon$ for some edge $v_iv_j\in E(\Gamma)$ then the corresponding term in $\rho\cdot\Delta^\varphi_{\alpha\theta}\rho$ is
\begin{align*}
    \frac{2\rho_i\rho_j\cos(\varphi_{ij}) - (\rho_i^2 + \rho_j^2)\cos(\alpha\theta_{ij})}{\sin(\alpha\theta_{ij})} & = \frac{2\rho_i\rho_j\cos(\varphi_{ij}) + \rho_i^2 + \rho_j^2 + O(\epsilon^2)}{\epsilon + O(\epsilon^3)}\\
     & = \frac{(\rho_i-\rho_j)^2}{\epsilon} + \big(1-\cos(\varphi_{ij})\big)\frac{2\rho_i\rho_j}{\epsilon} + O(\epsilon).
\end{align*}
The only way this term does not approach $+\infty$ is if $\rho_i=\rho_j=0$.

As $\alpha\to\pi/\theta_{\max}$, the quadratic form $\rho\cdot\Delta^\varphi_{\alpha\theta}\rho$ stays bounded only if $\rho_i=\rho_j=0$ for each edge $v_iv_j$ with $\theta_{ij}=\theta_{\max}$. In other words, $\rho_i$ must be 0 at any vertex $v_i$ incident to an edge $e$ with $\theta(e)=\theta_{\max}$. This defines a subspace $W$ on which the quadratic form remains bounded. For all $\rho\not\in W$,
\[
    \lim_{\alpha\to\pi/\theta_{\max}}\rho\cdot\Delta^\varphi_{\alpha\theta}\rho = +\infty.
\]
We can thus define a quadratic form $Q$ on $W$ via
\[
    Q(\rho) = \lim_{\alpha\to\pi/\theta_{\max}}\rho\cdot\Delta^\varphi_{\alpha\theta}\rho.
\]

\begin{corollary}
    The total number of degrees $0<\alpha<\pi/\theta_{\max}$ that admit balanced homogeneous harmonic maps $f:C(\Gamma,\theta)\to C(\Gamma,\varphi)$, counted with multiplicity, is the sum of the dimension $dim(W^\perp)$ and the number of positive eigenvalues of the quadratic form $Q$ defined on $W$. 
\end{corollary}

\begin{remark}
    If $\Delta^\varphi_{\alpha\theta}$ has a non-trivial kernel, the \emph{multiplicity} of $\alpha$ is the dimension of that kernel.
\end{remark}

\begin{proof}
    Following Theorem~\ref{2-d eigenvalues}, the only time that $\Delta^\varphi_{\alpha\theta}$ can have a non-trivial kernel is when one of the eigenvalues $\lambda_j(\alpha)$ crosses 0. According to Theorem~\ref{neg def} the matrices $\Delta^\varphi_{\alpha\theta}$ have only negative eigenvalues for $\alpha$ small, so we need only count the number of positive eigenvalues as $\alpha$ approaches $\pi/\theta_{\max}$. Any eigenvalue of $\Delta^\varphi_{\alpha\theta}$ that approaches $+\infty$ as $\alpha\to\pi/\theta_{\max}$ is certainly positive, and the remaining eigenvalues of $\Delta^\varphi_{\alpha\theta}$ approach the eigenvalues of the quadratic form $Q$ on $W$ as described above.
\end{proof}

\subsection{The singular degree $\alpha=\pi/\theta_{\max}$}

In stark contrast to Section~\ref{S euclidean singular}, balanced homogeneous harmonic maps $f:C(\Gamma,\theta)\to C(\Gamma,\varphi)$ of the singular degree $\alpha=\pi/\theta_{\max}$ are quite rare. First we must describe homogeneous harmonic maps in this case. Let $\Sigma\subset\Gamma$ be the subgraph consisting of all those edges $e\in E(\Gamma)$ with $\alpha\theta(e)=\pi$, along with their incident vertices.

\begin{lemma}\label{2-d singular determinism}
    A homogeneous harmonic map $f:C(\Gamma,\theta)\to C(\Gamma,\varphi)$ of degree $\alpha=\pi/\theta_{\max}$ is uniquely determined by a function $\rho:V(\Gamma)\backslash V(\Sigma)\to\hl$, along with a function
    \[
        \nu:\{(v,e)\in V(\Sigma)\times E(\Sigma) \vert v\in e\}\to\R
    \]
    subject to the constraints
    \[
        \nu(v,e)\geq\cos(\varphi(e))\nu(w,e)\quad\text{and}\quad \nu(w,e)\geq\cos(\varphi(e))\nu(v,e)\qquad\text{for all }e=vw\in E(\Gamma).
    \]
\end{lemma}

\begin{proof}
    If $f:C(\Gamma,\theta)\to C(\Gamma,\varphi)$ has degree $\alpha=\pi/\theta_{\max}$, let $\rho:V(\Gamma)\to\hl$ be defined by $f(v,1) = (v,\rho(v))$. Any vertex $v\in V(\Sigma)$ is incident to an edge $e$ with $\alpha\theta(e)=\pi$, so according to Proposition~\ref{maps between sectors} $\rho(v)=0$. Thus we may replace $\rho$ with its restriction to $V(\Gamma)\backslash V(\Sigma)$. Using Proposition~\ref{maps between sectors}, such a $\rho$ uniquely determines the behavior of $f$ on each face $C(e)$ corresponding to $e\in E(\Gamma)\backslash E(\Sigma)$.
    
    Now we turn to those edges $e\in E(\Sigma)$. According to Proposition~\ref{maps between sectors}, the harmonic map $f:C(\Gamma,\theta)\to C(\Gamma,\varphi)$ can be represented in coordinates in the face $C(e)$ by the map
    \[
        u(r,\theta) = (x,y) = \Big(c_1r^\alpha\sin(\alpha\theta),c_2r^\alpha\sin(\alpha\theta)\Big).
    \]
    This is a map $S_{\theta(e)}\to S_{\varphi(e)}$, and we are using polar coordinates in the domain and rectangular coordinates in the target.
    
    If the positive $x$-axis corresponds to a vertex $v_1\in V(\Gamma)$ and the other edge of $S_{\theta(e)}$ (or $S_{\varphi(e))}$) corresponds to $v_2\in V(\gamma)$, consider $\nu_1=\nu(v_1,e)$ and $\nu_2=\nu(v_2,e)$ to be the normal derivatives used in the balancing condition of Definition~\ref{2-d balancing}. That is,
    \[
        \nu_1 = \frac{\partial x}{\partial\theta}(1,0) = c_1\alpha
    \]
    and
    \[
        \nu_2 = -\frac{\partial}{\partial\theta}\vert_{r=1,\theta=\theta(e)}\Big(x\cos\varphi(e)+y\sin\varphi(e)\Big) = \alpha\big(c_1\cos\varphi(e)+c_2\sin\varphi(e)\big).
    \]
    One can easily solve for $c_1$ and $c_2$.
    \[
        c_1 = \frac{\nu_1}{\alpha}\qquad\text{and}\qquad c_2 = \frac{\nu_2-\nu_1\cos\varphi(e)}{\alpha\sin\varphi(e)}.
    \]
    
    Thus the map $\nu$ uniquely determines the map $f$ in each face $C(e)$ corresponding to $e\in E(\Sigma)$. The only restriction on $\nu$ is that the image of $u$ must lie in $S_{\varphi_0}$. The conditions are very similar to those found in Proposition~\ref{maps between sectors}, namely $c_2\geq0$ and $c_1\geq c_2\cot(\varphi_0)$. Using the above equations to express $c_1$ and $c_2$ in terms of $\nu_1$ and $\nu_2$ yields precisely
    \[
        \nu_2\geq\nu_1\cos\varphi_0\qquad\text{and}\qquad \nu_1\geq\nu_2\cos\varphi_0.
    \]
\end{proof}

Now we discuss the balancing condition. On vertices $v_i\in V(\Gamma)\backslash V(\Sigma)$ one must still have
\[
    \sum_{v_j\sim v_i}\frac{\cos(\varphi_{ij})\rho_j-\cos(\alpha\theta_{ij})\rho_i}{\sin(\alpha\theta_{ij})} = 0,
\]
just as in Theorem~\ref{euclidean balanced}. And on those vertices $v_i\in V(\Sigma)$ one has $\rho_i=0$ and the balancing condition reads
\[
    \sum_{v_iv_j\in E(\Sigma)}\nu(v_i,v_iv_j) + \sum_{v_iv_j\in E(\Gamma)\backslash E(\Sigma)}\frac{\cos(\varphi_{ij})}{\sin(\alpha\theta_{ij})}\rho_j = 0.
\]

It seems unlikely to find a map balanced at just the vertices in $V(\Gamma)\backslash V(\Sigma)$ without taking $\rho\equiv 0$. Since the eigenvalue of $\Delta^\varphi_{\alpha\theta}$ are strictly increasing with $\alpha$ the chance that one passes through 0 precisely at a particular value of $\alpha$ seems low. Despite a similar concern in Section~\ref{S euclidean singular}, we described an abundance of balanced harmonic maps in Proposition~\ref{abundance} by choosing $\rho\equiv0$. Unfortunately such a result does not hold in the current situation.

\begin{theorem}\label{non-existence}
    The only balanced homogeneous harmonic map $f:C(\Gamma,\theta)\to C(\Gamma,\varphi)$ of degree $\alpha=\pi/\theta_{\max}$, determined by $\rho\equiv 0$ and some function $\nu$ as in Lemma~\ref{2-d singular determinism}, is the trivial map.
\end{theorem}

\begin{proof}
    The map $f:C(\Gamma,\theta)\to C(\Gamma,\varphi)$ determined by the functions $\rho\equiv 0$ and $\nu$ must satisfy the conditions listed in Lemma~\ref{2-d singular determinism}. For an edge $e=vw\in E(\Sigma)$ this says
    \[
        \nu(v,e) \geq \nu(w,e)\cos(\varphi(e)) \qquad\text{and}\qquad \nu(w,e)\geq\nu(v,e)\cos(\varphi(e)).
    \]
    Adding these together gives
    \[
        \nu(v,e)+\nu(w,e) \geq \Big(\nu(v,e)+\nu(w,e)\Big)\cos(\varphi(e)).
    \]
    Since $\cos(\varphi(e))$ cannot equal 1, we must have
    \[
        \nu(v,e)+\nu(w,e) \geq 0.
    \]
    
    Summing over all edges of $\Sigma$ we have
    \begin{align*}
        0 & \leq \sum_{e\in E(\Gamma)}\sum_{v\in e}\nu(v,e)\\
         & = \sum_{v\in V(\Gamma)}\sum_{e\ni v}\nu(v,e) = 0.
    \end{align*}
    The vanishing of the reordered sum follows from the balancing condition at the vertices of $\Sigma$. Thus each inequality from Lemma~\ref{2-d singular determinism} must in fact be an equality. So on the edge $e=vw\in E(\Gamma)$,
    \[
        \nu(v,e) = \nu(w,e)\cos(\varphi(e)) = \nu(v,e)\cos^2(\varphi(e)).
    \]
    Since $\cos(\varphi(e))$ never equals 1, $\nu(v,e) = 0$ for each pair $(v,e)$.
\end{proof}

As an immediate consequence we can extend the result of Proposition~\ref{constant theta and phi}, which discussed the case when $\theta(e)=\theta_0$ and $\varphi(e)=\varphi_0$ for all $e\in E(\Gamma)$. That Proposition insisted the only balanced harmonic maps had degree $\alpha = \varphi_0/\theta_0$, but did not discuss the possible singular degree $\alpha = \pi/\theta_0$.

\begin{corollary}
     Let $\theta(e)=\theta_0$ and $\varphi(e)=\varphi_0$ for all $e\in E(\Gamma)$. Then there is no non-trivial balanced homogeneous harmonic map $f:C(\Gamma,\theta)\to\ C(\Gamma,\varphi)$ of degree $\alpha=\pi/\theta_0$.
\end{corollary}

\begin{proof}
    If $\theta(e)=\theta_0$ for all $e\in E(\Gamma)$ and $\alpha = \pi/\theta_0$ then the subgraph $\Sigma$ is all of $\Gamma$ itself. According to Lemma~\ref{2-d singular determinism} a homogeneous harmonic map $f:C(\Gamma,\theta)\to C(\Gamma,\varphi)$ is determined by just a function $\nu$. But now Theorem~\ref{non-existence} says that in order for $f$ to be balanced $\nu$ must vanish. In other words, $f$ must be the trivial map that sends all of $C(\Gamma,\theta)$ to the vertex of $C(\Gamma,\varphi)$.
\end{proof}

If one could balance the vertices in $V(\Gamma)\backslash V(\Sigma)$, then balancing the vertices in $V(\Sigma)$ amounts to choosing a function $\nu$ from a $2\#E(\Sigma)$-dimensional space subject to $\#V(\Sigma)$ balancing conditions. Such a $\nu$ certainly exists since $2\#E(\Sigma)\geq\#V(\Sigma)$ (in fact the inequality is strict if $\Sigma$ is not a disjoint collection of edges) and the $\#V(\Sigma)$ balancing conditions are independant (no two of them use the same value of $\nu$). But as was the case in Proposition~\ref{non-existence} the extra inequalities restricting $\nu$ are unlikely to be satisfied.

\section{Collapsing cones}\label{S puncture}

This section deals with the limits of the cones $C(\Gamma,t\theta)$ as $t\to0$. These spaces can also be seen as tangent cones to the ideal hyperbolic simplicial complexes of \cite{me-gras-1} at punctured vertices.

The na\"ive approach is to simply let $t$ tend to 0 in $C(\Gamma,t\theta)$. Each face of $C(\Gamma)$ will collapse to a half-line, its bounding edges merging. If $\Gamma$ is a connected graph then all the faces and edges of $C(\Gamma)$ will collapse down to a single half-line, $\hl$.

The results from Section~\ref{S euclidean target} would turn into a search for a homogeneous harmonic map $f:\hl\to\R$, which must be simply a degree 1 map $f(x)=ax$. Such a map also generalizes the results of Section~\ref{S k-pod target}, mapping the collapsed domain to a single edge of the $k$-pod target. If the domain collapses in Section~\ref{S singular target} in this way then a similar map can be constructed into a single edge or face of the target. And if only the target collapses then one searches for a balanced homogeneous harmonic function $f:C(\Gamma,\theta)\to\hl$. Unfortunately no such map exists, as the following result will show.

\begin{proposition}
    If $f:C(\Gamma,\theta)\to\R$ is a balanced homogeneous harmonic map of degree $\alpha>0$ then the average of $f$ on any ball centered at the vertex of $C(\Gamma,\theta)$ is 0.
\end{proposition}

An immediate consequence is that a balanced homogeneous harmonic function $f:C(\Gamma,\theta)\to\hl$ must be identically 0.

\begin{proof}
    Just as in Section~\ref{S eigenvalues}, the map $f:C(\Gamma,\theta)\to\R$ can be written as $f(x,t) = \rho(x)t^\alpha$ for some $\rho:\Gamma\to\R$. After choosing coordinate functions $\psi_e:e\to[0,\theta(e)]$ for each $e\in E(\Gamma)$ and defining $\rho_e=\rho\circ \psi_e^{-1}$, $\rho$ is a solution to the problem
    \[
        \rho''_e=-\alpha^2\rho_e\text{ for all }e\in E(\gamma)\qquad\text{and}\qquad \sum_{v\in e}(-1)^{\psi_e(v)/\theta(e)}\rho'_e(\psi(e)) = 0\text{ for all }v\in V(\Gamma).
    \]

    The average of $f$ over a ball centered at the vertex of $C(\Gamma,\theta)$ can be computed as follows.
    \begin{align*}
        \iint_{B_r(0)} fdA & = \sum_{e\in E(\Gamma)}\iint_{C_r(0)\cap C(e)}fdA\\
         & = \sum_{e\in E(\Gamma)}\int_0^r\int_0^{\theta(e)}\rho_e(x)t^\alpha tdxdt\\
         & = \frac{-r^{\alpha+2}}{\alpha^2(\alpha+2)}\sum_{e\in E(\Gamma)}\int_0^{\theta(e)}\rho''_e(x)dx\\
         & = \frac{-r^{\alpha+2}}{\alpha^2(\alpha+2)}\sum_{e\in E(\Gamma)}\Big(\rho'_e(\theta(e))-\rho'_e(0)\Big)\\
         & = \frac{r^{\alpha+2}}{\alpha^2(\alpha+2)}\sum_{v\in V(\Gamma)}\sum_{e\ni v}(-1)^{\psi_e(v)/\theta(e)}\rho'_e(\psi(v))\\
         & = 0.
    \end{align*}
\end{proof}

\begin{remark}
    This result can also be seen as the $L^2$-orthogonality between the eigenfunction $\rho$ with eigenfunvalue $\alpha$ and the eigenfunction $\rho_0$ with $\rho_0(x)=1$, whose eigenvalue is $0$. Similar integrations by parts in fact imply that all eigenfunctions with distinct eigenvalues are orthogonal in $L^2$.
\end{remark}

Evidently there is not much happening when we allow the metrics on our cones to collapse, so more care must be taken. As we scale the angle measures by $t\to0$, we will dilate the cones so that the curves originally at distance 1 from the vertex maintain a constant length. In a face $S_{\theta_0}$ this can be achieved with the following 4 step procedure:
\begin{enumerate}
    \item In polar coordinates map $(r,\theta)\mapsto(r^t,\theta)$. This does not change $S_{\theta_0}$ as a set, but is useful for defining coefficients later on.
    \item In polar coordinates map $(r,\theta)\mapsto(r,t\theta)$. This maps $S_{\theta_0}$ to $S_{t\theta_0}$.
    \item In rectangular coordinates map $(x,y)\mapsto(x-1,y)$. This has the effect of moving the curve originally at distance 1 from the vertex so it now passes through the origin.
    \item Dilate the region $S_{t\theta_0}$ by a factor of $1/t$. This moves the vertex to the point $(-1/t,0)$ in recangular coordinates, and rescales the curve now passing through the origin to have length $\theta_0$ as it did at the start.
\end{enumerate}
This transformation can be accomplished much more concisely in complex coordinates by the map
\[
    z \mapsto \frac{z^t-1}{t}.
\]
In the limit as $t\to0$ this transformation converges to $z\mapsto\log(z)$, and the image of $S_{\theta_0}$ under this transformation is a strip, described in rectangular coordinates by
\[
    R_{\theta_0} = \{(x,y)\in\R^2\vert 0\leq y \leq \theta_0\}.
\]

Thus we define the space
\[
    C^*(\Gamma,\theta)=\Gamma\times\R
\]
with a product metric. The factor $\R$ is endowed with its usual Euclidean metric, and $\Gamma$ is endowed with a path metric where each edge $e\in E(\Gamma)$ has length $\theta(e)$. Note that in the definition of $C^*(\Gamma,\theta)$ there is absolutely no difficulty if the function $\theta$ takes values larger than $\pi$.

\subsection{Balanced harmonic functions into $\R$}

A homogeneous map $u:S_{\theta_0}\to\R$ can be described in polar coordinates as $u(r,\theta)=r^\alpha\rho(\theta)$ for some $\alpha>0$. The complex exponential transforms the rectangular coordinates $(x,y)$ to the polar coordinates $(r=e^{x},\theta=y)$. So precomposing with the complex exponential gives a map $v = u\circ\exp:R_{\theta_0}\to\R$ given in rectangular coordinates by
\[
    v(x,y) = u(e^x,y) = e^{\alpha x}\rho(y).
\]
When we move to functions $f:C^*(\Gamma,\theta)\to\R$ this suggests that we should consider the class of separated functions, a broader class than homogeneous functions.

\begin{definition}\label{separated}
    A map $f:C^*(\Gamma,\theta)\to\R$ is separated if there are functions $\rho:\Gamma\to\R$ and $\tau:\R\to\R$ so that
    \[
        f(x,t) = \rho(x)\tau(t).
    \]
\end{definition}

Since the transformation $z\mapsto\log(z)$ is conformal on the interior of $S_{\theta_0}$, a homogeneous harmonic function $u(r,\theta)=r^\alpha\rho(\theta)$ is transformed to a \emph{harmonic} function $v(x,y)=e^{\alpha x}\rho(y)$. This suggests that the search for balanced harmonic functions $v:R_{\theta_0}\to\R$ will be equivalent to the search in Section~\ref{S euclidean target}.

Of course the class of separated harmonic functions is larger than those maps described above, so some care must be taken. Using coordinates $(x,t)$ on $C^*(\Gamma,\theta) = \Gamma\times\R$, a separated function has the form
\[
    f(x,t) = \rho(x)\tau(t).
\]
By the theory of separation of variables, such a function is harmonic if both $\rho$ and $\tau$ are eigenfunctions of the second derivative operator with opposite eigenvalues. Choosing isometries $\psi_e:e\to[0,\theta(e)]$ for each $e\in E(\Gamma)$, let $\rho_e = \rho\circ\psi_e^{-1}$. Then for each $e$ we must have $\tau''=\lambda \tau$ and $\rho''_e=-\lambda \rho_e$ for some $\lambda\in\R$.

The eigenfunctions $u''=cu$ are given by
\[
    u(x) =  \begin{cases}
        A\cosh(\sqrt{c}x)+B\sinh(\sqrt{c}x), & c>0\\
        A+Bx, & c=0\\
        A\cos(\sqrt{-c}x)+B\sin(\sqrt{-c}x), & \lambda<0.
    \end{cases}
\]
Since $\tau$ is independent of $e\in E(\gamma)$, the eigenvalue $-\lambda$ for $\rho_e$ must also be independent of $e$.

Using the isometries $\psi_e:e\to[0,\theta(e)]$ and the representations $\rho_e=\rho\circ\psi_e^{-1}$ the balancing condition of Definition~\ref{euclidean balancing} can be generalized to separated functions as follows.

\begin{definition}\label{balanced separated}
    A separated function $f(x,t)=\rho(x)\tau(t)$ from $C^*(\Gamma,\theta)$ to $\R$ is balanced along the edge $\{v\}\times\R\subset C^*(\Gamma,\theta)$ corresponding to $v\in V(\Gamma)$ if
    \[
        \sum_{e\ni v}(-1)^{\psi_e(v)/\theta(e)}\rho'_e(\psi_e(v)) = 0.
    \]
    The function $f$ is balanced if it is balanced along each edge of $C^*(\Gamma,\theta)$.
\end{definition}

Now we are in a position to fully understand balanced separated harmonic functions and their relationship to the balanced homogeneous harmonic functions of Section~\ref{S euclidean target}.

\begin{theorem}\label{balanced separated functions}
    Let $f:C^*(\Gamma,\theta)\to\R$ be a separated harmonic function, $f(x,t) = \rho(x)\tau(t)$. The function $f$ is balanced if and only if one of the following occurs.
    \begin{enumerate}
        \item $f(x,t) = A+Bt$, or
        \item there is a number $\alpha>0$ so that $g(x,t)=\rho(x)t^\alpha:C(\Gamma,\theta)\to\R$ is a balanced homogeneous harmonic function of degree $\alpha$.
    \end{enumerate}
    In the second case $\tau(t) = A\cosh(\alpha t)+B\sinh(\alpha t) = Ce^{\alpha t}+De^{-\alpha t}$.
\end{theorem}

\begin{remark}
    At the beginning of this subsection we saw that the functions coming from balanced homogeneous functions $\tilde{f}:C(\Gamma,\theta)\to\R$ corresponded to $f=\rho\tau$ with $\tau(t)=Ce^{\alpha t}$. We also develop the additional solutions $f(x,t)=A+Bt$, which correspond to $\alpha = 0$. In fact the balanced homogeneous function $\tilde{f}(x,t) = 1$ of degree $0$ gives rise to the balanced separated function $f(x,t) = 1$, corresponding to $B=0$.

    See also that the functions $f(x,t)=A$ and $f(x,t) = \rho(x)e^{\alpha t}$ are the only ones with $\lim_{t\to-\infty}f(x,t) = 0$.
\end{remark}

\begin{proof}
    Let $f:C^*(\Gamma,\theta)\to\R$ be a balanced separated harmonic function. With isometries $\psi_e:e\to[0,\theta(e)]$ and representations $\rho_e=f\circ\psi_e^{-1}:[0,\theta(e)]\to\R$ in coordinates, the discussion above implies that the functions $\rho_e$ satisfy
    \[
        \rho''_e = -\lambda \rho_e\text{ for all }e\in E(\Gamma) \qquad\text{and}\qquad \sum_{e\ni v}(-1)^{\psi_e(v)/\theta(e)}\rho'_e(\psi(v)) = 0\text{ for all }v\in V(\Gamma).
    \]
    This is precisely the eigenvalue problem studied in Section~\ref{S eigenvalues}. According to Theorem~\ref{rayleigh}, the solutions $\rho$ to this problem are critical points of the Rayleigh quotient
    \[
        R[\rho] = \frac{\sum_e \norm{\rho'_e}^2}{\sum_e\norm{\rho_e}^2},
    \]
    where $\norm{-}$ indicates the $L^2$ norm on an interval. Moreover the eigenvalue $-\lambda$ of the solution $\rho$ is precisely $-R[\rho]$, so $\lambda\geq 0$. Let $\alpha\geq 0$ be such that $\lambda = \alpha^2$.

    If $\alpha = 0$, then $\rho_e(x) = A_e+B_ex$ for each $e\in E(\Gamma)$, i.e. $\rho$ is edge-wise linear. Just as in Proposition~\ref{degree 0}, though, $\rho$ is only balanced if it is constant. In this case $\rho\equiv c$ for some constant $c$. Since $\lambda = 0$ the function $\tau(t)$ solves $\tau'' = 0$, so $\tau$ too is linear. say $\tau(t) = a+bt$. In this case
    \[
        f(x,t) = \rho(x)\tau(t) = ac+bct = A+Bt.
    \]

    If $\alpha>0$ then according to Section~\ref{S eigenvalues} the function $\rho:\Gamma\to\R$ fits into a balanced homogeneous harmonic function $g:C(\Gamma,\theta)\to\R$ via
    \[
        g(x,t) = \rho(x)t^\alpha.
    \]
    In this case $\tau''=\alpha^2\tau$, so $\tau(t) = A\cosh(\alpha t)+B\sinh(\alpha t) = Ce^{\alpha t}+De^{-\alpha t}$.
\end{proof}

\subsection{Balanced harmonic maps between cones}

Just as in Section~\ref{S singular target} we will consider only maps $f:C^*(\Gamma,\theta)\to C^*(\Gamma,\varphi)$ between spaces with the same combinatorial structure, and moreover only those maps that map the edges $v\times\R$ and faces $e\times R$ of $C^*(\Gamma)$ to themselves. If $\psi_e:e\to[0,\theta(e)]$ is an isometry for the metric $\theta$, then we define a map $\Psi_e:e\times\R\to R_{\theta(e)}$ by
\[
    \Psi_e(x,t) = \big(t,\psi_e(x)\big).
\]
Taking isometries $\psi^\theta_e:e\to[0,\theta(e)]$ and $\psi^\varphi_e:e\to[0,\varphi(e)]$ for each $e\in E(\Gamma)$, ensure that $\psi^\theta_e(v)/\theta(e) = \psi^\varphi_e(v)/\varphi(e)$ for each $v\in e$, i.e. $\psi^\theta_e$ and $\psi^\varphi_e$ induce the same orientation on $e$. Then the map $f\vert_{e\times\R}$ can be represented in coordinates by
\[
    u_e = \Psi^\varphi_e\circ f\circ (\Psi^\theta_e)^{-1}:R_{\theta(e)}\to R_{\varphi(e)}.
\]
But only some such maps are limits of homogeneous maps under collapsing metrics.

A homogeneous map $u:S_{\theta_0}\to S_{\varphi_0}$ can be described in polar coordinates as $u(r,\theta)=\big(r^\alpha\rho(\theta),\phi(\theta)\big)$ for some $\alpha>0$. In Section~\ref{S singular target} we used rectangular coordinates on the target to describe such maps, but the translation to polar coordinates is not difficult. The complex exponential transforms the rectangular coordinates $(x,y)$ to the polar coordinates $(r=e^{x},\theta=y)$, so conjugating with the complex exponential gives a map $v = \log\circ u\circ\exp:R_{\theta_0}\to R_{\varphi_0}$ given in rectangular coordinates by
\[
    v(x,y) = \log u(e^x,y) = \log\Big(r^{\alpha x}\rho(y)e^{i\phi(y)}\Big) = \big(\alpha x + \log\rho(y), \phi(y)\big).
\]

Though the presence of the complex logarithm above turns a harmonic $u:S_{\theta_0}\to S_{\varphi_0}$ into a map $v:R_{\theta_0}\to R_{\varphi_0}$ that is not necessarily harmonic, some (non-harmonic) choices of homogeneous $u$ still result in harmonic $v$. As both $R_{\theta_0}$ and $R_{\varphi_0}$ are Euclidean, the map $v$ is harmonic if and only if both $\alpha x+\log\rho(y)$ and $\phi(y)$ are. Unfortunately this only happens when $\rho(y)=e^{\beta y+a}$ and $\phi(y) = \gamma y+b$ for some constants $\beta,\gamma$, in which case $f$ is linear.

When we move to functions $f:C^*(\Gamma,\theta)\to C^*(\Gamma,\varphi)$ we will consider only maps that send each line $\{x\}\times\R\subset C^*(\Gamma,\theta)$ to a corresponding line $\{x'\}\times\R\subset C^*(\Gamma,\varphi)$. Representing the map $f$ in coordinates in the face $e\times\R$ by the map $u_e:R_{\theta(e)}\to R_{\varphi(e)}$, the maps $u_e$ take the form
\[
    u_e(x,y) = \big(\rho_e(x,y),z_e(y)\big).
\]

For $f$ to be harmonic each $u_e$ must be harmonic. In particular $z''_e=0$, so $z_e$ is linear. And in order to map each edge $v\times\R$ to itself we must in fact have
\[
    z_e(y) = \frac{\varphi(e)}{\theta(e)}y.
\]
And the balancing condition of Definition~\ref{2-d balancing} can be generalized as follows.

\begin{definition}
    The map $f:C^*(\Gamma,\theta)\to C^*(\Gamma,\varphi)$ is balanced along the edge $v\times\R$ corresponding to $v\in V(\Gamma)$ if
    \[
        \sum_{e\ni v}(-1)^{\psi^\theta_e(v)/\theta(e)}\frac{\partial \rho_e}{\partial y}(\Psi^\theta_e(v,t)) = 0\text{ for all }t\in\R.
    \]
    And the map $f$ is balanced if it is balanced along each edge of $C^*(\Gamma,\theta)$.
\end{definition}

Unlike in the case of functions to $\R$, the only harmonic maps $f:C^*(\Gamma,\theta)\to C^*(\Gamma,\varphi)$ that came from maps $C(\Gamma,\theta)\to C(\Gamma,\varphi)$ were the piecewise linear ones. And in fact there are very few \emph{balanced} piecewise lienar maps.

\begin{proposition}\label{final}
    If $f:C^*(\Gamma,\theta)\to C^*(\Gamma,\varphi)$ is a balanced harmonic map whose restriction to each face $e\times\R$ is represented by a linear map $u_e:R_{\theta(e)}\to R_{\varphi(e)}$, then each $u_e$ has the form
    \[
        u_e(x,y) = \left(ax+b,\frac{\varphi(e)}{\theta(e)}y\right)
    \]
    for some constants $a,b$ that are independent of $e$. In particular, $\rho_e$ is independent of $y$.
\end{proposition}

\begin{proof}
    If each function $\rho_e$ is linear, then write $\rho_e(x,y) = a_ex+b_ey+c_e$. If $f(v,0)=(v,\rho(v))$ for some function $\rho:V(\Gamma)\to\R$, then
    \[
        \frac{\partial \rho_e}{\partial y} = b_e = \frac{\rho\circ(\psi^\theta_e)^{-1}(\theta(e))-\rho\circ(\psi^\theta_e)^{-1}(0)}{\theta(e)}.
    \]
    The balancing condition at a vertex $v\in V(\Gamma)$ then reads
    \begin{align*}
        0 & = \sum_{e\ni v}(-1)^{\psi^\theta_e(v)/\theta(e)}\frac{\rho\circ(\psi^\theta_e)^{-1}(\theta(e))-\rho\circ(\psi^\theta_e)^{-1}(0)}{\theta(e)}\\
         & = \sum_{w\sim v}\frac{\rho(w)-\rho(v)}{\theta(vw)}.
    \end{align*}
    Thus the balancing condition for $f$ says that $\rho$ is in the kernel of the edge-weighted graph Laplacian with edge weights $w(e)=\frac{1}{\theta(e)}$. But the kernel of such an operator is spanned by the vector $\rho=\mathbb{1}$ with $\rho(v)=1$ for all $v\in V(\Gamma)$.

    So if $f$ is balanced then $\rho$ must be constant, i.e. each $b_e=0$. Then $\rho_e(x,y) = a_ex+c_e$. If $e\cap e' = v$, then we must also have $\rho_e(\Psi^\theta_e(v,t))=\rho_{e'}(\Psi^\theta_{e'}(v,t))$, i.e. $a_et+c_e=a_{e'}t+c_{e'}$. The connectedness of $\Gamma$ then implies that $a_e$ and $c_e$ are independent of $e$, so let $a=a_e$ and $b=c_e$ to find the result claimed.
\end{proof}

\begin{remark}
The above result can then be realized as a companion to case 1 of Theorem~\ref{balanced separated functions}. In fact in general, any balanced separated harmonic function $g:C^*(\Gamma,\theta)\to\R$ from Theorem~\ref{balanced separated functions}, or any linear combination or convergent series of such separated solutions, gives a balanced harmonic map $f:C^*(\Gamma,\theta)\to C^*(\Gamma,\varphi)$, for any $\varphi$, by the formula
\[
    f(x,t) = \left(x,g(x,t)\right).
\]
However the face-wise linear maps $f$ described in Proposition~\ref{final} are the only ones for which $\abs{\nabla f}^2$ is bounded on $C^*(\Gamma,\theta)$.
\end{remark}

\end{document}